\documentclass[11pt]{article}

\usepackage{amssymb}

\usepackage[english]{babel}
\usepackage{amsmath}
\usepackage{amssymb}
\usepackage{graphicx}
\usepackage{color}
\usepackage{xcolor}
\usepackage{booktabs}
\usepackage{array}   
\usepackage{afterpage}
\usepackage{subfig}
\usepackage[font={small,it}]{caption}

\usepackage{graphicx}
\usepackage{color}
\usepackage{amsmath,amsthm,amsfonts,epsfig,setspace}
\numberwithin{equation}{section}

\newcommand{\ds}{\displaystyle}
\def\nm{\noalign{\medskip}}
\newtheorem{thm}{Theorem}[section]
\newtheorem{rmk}{Remark}[section]

\newtheorem{lem}{Lemma}[section]

 \def\p{\partial}
\def \Vh0{\stackrel{\circ}{V}_h}

\def\l{\label}  \def\f{\frac}  

\def\l|{\left|}
\def\r|{\right|}

\newcommand{\R}{\mathbb{R}}
\newcommand{\C}{\mathbb{C}}

\newcolumntype{L}{>{$}l<{$}}
\newcolumntype{C}{>{$}c<{$}}
\newcolumntype{R}{>{$}r<{$}}

\newcommand{\lc}
{\mathrel{\raise2pt\hbox{${\mathop<\limits_{\raise1pt\hbox
{\mbox{$\sim$}}}}$}}}

\newcommand{\gc}
{\mathrel{\raise2pt\hbox{${\mathop>\limits_{\raise1pt\hbox{\mbox{$\sim$}}}}$}}}

\newcommand{\ec}
{\mathrel{\raise2pt\hbox{${\mathop=\limits_{\raise1pt\hbox{\mbox{$\sim$}}}}$}}}

\newcommand{\vertfig}[2][]{%
  \begin{minipage}{6in}\subfloat[#1]{#2}\end{minipage}}

\def\be{\begin{equation}} \def\ee{\end{equation}}

\def\bea{\begin{eqnarray}}  \def\eea{\end{eqnarray}}

\def\beas{\begin{eqnarray*}} \def\eeas{\end{eqnarray*}}

\def\bn{\begin{enumerate}} \def\en{\end{enumerate}}

\def\bd{\begin{description}} \def\ed{\end{description}}

\title{Minnaert resonances for acoustic waves in bubbly media\thanks{\footnotesize This work was supported  by the
ERC Advanced Grant Project MULTIMOD--267184. Hyundae Lee was supported by NRF-2015R1D1A1A01059357 grant.  Hai Zhang was supported by a startup fund from HKUST.}}

\author{
Habib Ammari\thanks{\footnotesize Department of Mathematics, 
ETH Z\"urich, 
R\"amistrasse 101, CH-8092 Z\"urich, Switzerland (habib.ammari@math.ethz.ch, brian.fitzpatrick@sam.math.ethz.ch, david.gontier@sam.math.ethz.ch ).} \and Brian Fitzpatrick\footnotemark[2] \and David Gontier\footnotemark[2] 
\and Hyundae Lee\thanks{Department of Mathematics, Inha University,  253 Yonghyun-dong Nam-gu,  Incheon 402-751,  Korea (hdlee@inha.ac.kr).}  
 \and Hai Zhang\thanks{\footnotesize 
Department of Mathematics, 
 HKUST,  Clear Water Bay, Kowloon, Hong Kong (haizhang@ust.hk).}}

\date{}
\begin{document}
\maketitle

\begin{abstract}
Through the application of layer potential techniques and Gohberg-Sigal theory we derive an original formula for the Minnaert resonance frequencies of arbitrarily shaped bubbles. We also provide a mathematical justification for the monopole approximation of scattering of acoustic waves by bubbles at their Minnaert resonant frequency. Our results are complemented by several numerical examples which serve to validate our formula in two dimensions.
\end{abstract}

\medskip

\bigskip

\noindent {\footnotesize Mathematics Subject Classification
(MSC2000): 35R30, 35C20.}



\section{Introduction} \label{sec-intro}

The purpose of this work is to understand acoustic wave propagation through a liquid containing bubbles. Our motivation is the use of bubbles in medical ultrasonic imaging as strong sound scatterers at particular frequencies known as Minnaert resonances.   Many interesting physical works have been devoted to the acoustic bubble problem; see, for instance,  \cite{calvo, leroy1, breathing, hwang, leroy2, unusual}.  Nevertheless, the characterization of the Minnaert resonances for arbitrary shaped bubbles has been a longstanding problem.  

In this paper we derive an original formula for the Minnaert resonances of bubbles of arbitrary shapes using layer potential techniques and Gohberg-Sigal theory \cite{akl}. Our formula can be generalized to multiple bubbles. We provide a mathematical justification for the monopole approximation and demonstrate the enhancement of the scattering in the far field at the Minnaert resonances. We show that there is a correspondence between bubbles in water and plasmonic nanoparticles in that both raise similar fundamental questions \cite{matias}. However, the mathematical formulation of Minnaert resonances is much more involved than the formulation of plasmonic resonances.   

The Minnaert resonance is a low frequency resonance in which the wavelength is much larger than the size of the bubble \cite{leroy1}. Our results in this paper have important applications. They can be used to show that at the Minnaert resonance it is possible to achieve superfocusing of acoustic waves or imaging of passive sources with a resolution beyond the Rayleigh diffraction limit \cite{hai, hai2}. Foldy's approximation applies and yields to the conclusion that the medium surrounding the source behaves like a high contrast dispersive medium \cite{Foldy}. As the dispersion is small, it has little effect on the superfocusing and superresolution phenomena. Effective equations for wave propagation in bubbly liquids have been derived in the low frequency regime where the frequency is much smaller than the Minneart resonance frequency \cite{caflish, caflish2,kargl}.  In this paper, however, we are more concerned with wave propagation in the resonant regime. 

The paper is organized as follows. In Section \ref{sec-minnaert}
we consider the scattering of acoustic waves in three dimensions by a single bubble and derive its Minnaert resonances in terms of its capacity, volume, and material parameters. In Section \ref{sec-3d-point-scatter} we derive the point scatterer approximation of the bubble in the far-field. In Section \ref{sec:numerics} we perform numerical simulations in two dimensions to illustrate the main findings of this paper.  The paper ends with some concluding remarks. In Appendix \ref{sec-appendix-3d}, we collect some useful asymptotic formulas for layer potentials in two and three dimensions.  Derivations of the two-dimensional Minnaert resonances are given in Appendix \ref{appendixB}. 

\section{The Minnaert resonance} \label{sec-minnaert}
We consider the scattering of acoustic waves in a homogeneous media by a bubble embedded inside. Assume that the bubble occupies a bounded and simply connected domain $D$ with $\p D \in C^{1, s}$ for some $0<s<1$. 
We denote by $\rho_b$ and $\kappa_b$ the density and the bulk modulus of the air inside the bubble, respectively. $\rho$ and $\kappa$ are the corresponding parameters for the background media $\R^3 \backslash D$. 
The scattering problem can be modeled by the following equations:
\be \label{eq-scattering}
\left\{
\begin{array} {ll}
&\ds \nabla \cdot \f{1}{\rho} \nabla  u+ \frac{\omega^2}{\kappa} u  = 0 \quad \mbox{in } \R^3 \backslash D, \\
\nm
&\ds \nabla \cdot \f{1}{\rho_b} \nabla  u+ \frac{\omega^2}{\kappa_b} u  = 0 \quad \mbox{in } D, \\
\nm
&\ds  u_{+} -u_{-}  =0   \quad \mbox{on } \partial D, \\
\nm
& \ds  \f{1}{\rho} \f{\p u}{\p \nu} \bigg|_{+} - \f{1}{\rho_b} \f{\p u}{\p \nu} \bigg|_{-} =0 \quad \mbox{on } \partial D,\\
\nm
&  u^s:= u- u^{i}  \,\,\,  \mbox{satisfies the Sommerfeld radiation condition.}
  \end{array}
 \right.
\ee
Here, $\partial/\partial \nu$ denotes the outward normal derivative and $|_\pm$ denote the limits from outside and inside $D$.  

We introduce four auxiliary parameters to facilitate our analysis:
\be \label{data1}
v = \sqrt{\frac{\rho}{\kappa}}, \,\, v_b = \sqrt{\frac{\rho_b}{\kappa_b}}, \,\, k=\omega v, \,\, k_b= \omega v_b.
\ee

We also introduce two dimensionless contrast parameters:
\be \label{data2}
\delta = \f{\rho_b}{\rho}, \,\, \tau= \f{k_b}{k}= \f{v_b}{v} =\sqrt{\f{\rho_b \kappa}{\rho \kappa_b}}. 
\ee

By choosing proper physical units, we may assume that the size of the bubble is of order 1 and that the wave speeds outside and inside the bubble are both of order 1. Thus the contrast between the wave speeds is not significant. We assume, however, that there is a large contrast in the bulk modulii. In summary, we assume that $\delta \ll 1 $ and $\tau= O(1)$. 

We use layer potentials to represent the solution to the scattering problem (\ref{eq-scattering}). Let the single layer potential $\mathcal{S}_{D}^{k}$ associated with $D$ and wavenumber $k$ be defined by
$$
\mathcal{S}_{D}^{k} [\psi](x) =  \int_{\p D} G(x, y, k) \psi(y) d\sigma(y),  \quad x \in  \p {D},
$$
where $$G(x, y, k)= - \f{e^{ik|x-y|}}{4 \pi|x-y|}$$ is the  Green function of the Helmholtz equation in $\R^3$, subject to the Sommerfeld radiation condition.
We also define the boundary integral operator $\mathcal{K}_{D}^{k, *}$ by
$$
\mathcal{K}_{D}^{k, *} [\psi](x)  = \int_{\p D } \f{\p G(x, y, k)}{ \p \nu(x)} \psi(y) d\sigma(y) ,  \quad x \in \p D. 
$$

Then the solution $u$ can be written as 
\be \label{Helm-solution}
u(x) = \left\{
\begin{array}{lr}
u^{in} + \mathcal{S}_{D}^{k} [\psi], & \quad x \in \R^3 \backslash \bar{D},\\
\mathcal{S}_{D}^{k_b} [\psi_b] ,  & \quad x \in {D},
\end{array}\right.
\ee
for some surface potentials $\psi, \psi_b \in  L^2(\p D)$. 
Using the jump relations for the single layer potentials, it is easy to derive that $\psi$ and $\psi_b$ satisfy the following system of boundary integral equations:
\be  \label{eq-boundary}
\mathcal{A}(\omega, \delta)[\Psi] =F,  
\ee
where
\[
\mathcal{A}(\omega, \delta) = 
 \begin{pmatrix}
  \mathcal{S}_D^{k_b} &  -\mathcal{S}_D^{k}  \\
  -\f{1}{2}Id+ \mathcal{K}_D^{k_b, *}& -\delta( \f{1}{2}Id+ \mathcal{K}_D^{k, *})
\end{pmatrix}, 
\,\, \Psi= 
\begin{pmatrix}
\psi_b\\
\psi
\end{pmatrix}, 
\,\,F= 
\begin{pmatrix}
u^{in}\\
\delta \f{\partial u^{in}}{\partial \nu}
\end{pmatrix}.
\]

One can show that the scattering problem (\ref{eq-scattering}) is equivalent to the boundary integral equations (\ref{eq-boundary}). 

Throughout the paper, we denote by $\mathcal{H} = L^2(\p D) \times L^2(\p D)$ and by $\mathcal{H}_1 = H^1(\p D) \times L^2(\p D)$, 
and use $(\cdot, \cdot)$ for the inner product in $L^2$ spaces and $|| \; ||$ for the norm in $\mathcal{H}$.  Here, $H^1$ is the standard Sobolev space. 
It is clear that $\mathcal{A}(\omega, \delta)$ is a bounded linear operator from $\mathcal{H}$ to $\mathcal{H}_1$, i.e.
$\mathcal{A}(\omega, \delta) \in \mathcal{L}(\mathcal{H}, \mathcal{H}_1)$. 

The resonance of the bubble in the scattering problem (\ref{eq-scattering}) can be defined as all the complex numbers $\omega$ with negative imaginary part such that there exists a nontrivial solution to the following equation:
\be  \label{eq-resonance}
\mathcal{A}(\omega, \delta)[\Psi] =0.
\ee
These can be viewed as the characteristic values of the operator-valued analytic function (with respect to $\omega$)
$\mathcal{A}(\omega, \delta)$.
We are interested in the quasi-static resonance of the bubble, or the resonance frequency at which the size of the bubble is much smaller than the wavelength of the incident wave outside the bubble. In some physics literature, this resonance is called the Minnaert resonance. Due to our assumptions on the bubble being of size order one, and the wave speed outside of the bubble also being of order one, this resonance should lie in a small neighborhood of the origin in the complex plane. 
In what follows, we apply the Gohberg-Sigal theory to find this resonance. 

We first look at the limiting case when $\delta =\omega=0$.
It is clear that

\be  \label{eq-A_0-3d}
\mathcal{A}_0:= \mathcal{A}(0, 0) = 
 \begin{pmatrix}
  \mathcal{S}_D &  -\mathcal{S}_D  \\
  -\f{1}{2}Id+ \mathcal{K}_D^{*}& 0
\end{pmatrix},
\ee
where, for $\psi \in L^2(\partial D)$ and $x \in \partial D$, 
$$ \begin{array}{lll}
\mathcal{S}_{D} [\psi](x) &=& \ds -  \frac{1}{4\pi}  \int_{\p D} \f{ \psi(y)}{|x-y|}d\sigma(y),\\
\nm
\mathcal{K}_{D}^{*} [\psi](x) & =&\ds - \frac{1}{4\pi} \int_{\p D } \f{(x-y)\cdot \nu_x}{|x-y|^3} \psi(y) d\sigma(y) . 
\end{array}
$$

Let $\chi_{\p D}$ denote the characteristic function of $\partial D$ and let $\mathcal{A}_0^*$ be the adjoint of $\mathcal{A}$.  
\begin{lem} We have
\begin{enumerate}
\item[(i)]
$ Ker (\mathcal{A}_0) = span\, \{ \Psi_0 \}$ where 
\[
\Psi_0 = \alpha_0 \begin{pmatrix}
    \psi_0\\
  \psi_0\end{pmatrix}
\]
with $\psi_0 = \mathcal{S}_D^{-1}[\chi_{\p D}]$ and the constant $\alpha_0$ being chosen such that $\|\Psi_0\|=1$;

\item[(ii)]
$ Ker (\mathcal{A}_0^*) = span\, \{ \Phi_0 \}$ where 
\[
\Phi_0 = \beta_0 \begin{pmatrix}
    0\\
  \phi_0\end{pmatrix}
\]
with $\phi_0 = \chi_{\p D}$ and the constant $\beta_0$ being chosen such that $\|\Phi_0\|=1$.
\end{enumerate}
\end{lem}

The above lemma shows that $\omega=0$ is a characteristic value for the operator-valued analytic function $\mathcal{A}(\omega, \delta)$. By the Gohberg-Sigal theory \cite{akl}, we can conclude the following result about the existence of the quasi-static resonance.
\begin{lem}
For any $\delta$, sufficiently small, there exists a characteristic value 
$\omega_0= \omega_0(\delta)$ to the operator-valued analytic function 
$\mathcal{A}(\omega, \delta)$
such that $\omega_0(0)=0$ and 
$\omega_0$ depends on $\delta$ continuously. This characteristic value is also the quasi-static resonance (or Minnaert resonance). 
\end{lem}

We next perform asymptotic analysis on the operator $\mathcal{A}(\omega, \delta)$. Using the results in Appendix \ref{sec-appendix-3d}, we can derive the following result. 

\begin{lem} In the space $\mathcal{L}(\mathcal{H}, \mathcal{H}_1)$, we have
\[
\mathcal{A}(\omega, \delta):=\mathcal{A}_0 + \mathcal{B}(\omega, \delta)
= \mathcal{A}_0 + \omega \mathcal{A}_{1, 0}+ \omega^2 \mathcal{A}_{2, 0}
+ \omega^3 \mathcal{A}_{3, 0} + \delta \mathcal{A}_{0, 1}+ \delta \omega^2\mathcal{A}_{2, 1} + O(\omega^4)+ O(\delta \omega^3),
\]
where 
\[
\mathcal{A}_{1, 0}=
\begin{pmatrix}
  \tau v\mathcal{S}_{D,1} &  -v\mathcal{S}_{D,1}  \\
  0& 0
\end{pmatrix},
\,\, \mathcal{A}_{2,0}= 
\begin{pmatrix}
  \tau^2 v^2\mathcal{S}_{D,2} &  -v^2\mathcal{S}_{D,2}  \\
  \tau^2 v^2\mathcal{K}_{D,2}& 0
\end{pmatrix},
\,\, \mathcal{A}_{3,0}= 
\begin{pmatrix}
  \tau^3 v^3\mathcal{S}_{D,3} &  -v^3\mathcal{S}_{D,3}  \\
 \tau^3 v^3\mathcal{K}_{D,3}& 0
\end{pmatrix},
\]
\[
\mathcal{A}_{0, 1}=
\begin{pmatrix}
0& 0\\
0 &  -(\f{1}{2}Id+ \mathcal{K}_{D}^*)
\end{pmatrix},
\,\,  \mathcal{A}_{2, 1}=
\begin{pmatrix}
0& 0\\
0 &  -v^2\mathcal{K}_{D,2}
\end{pmatrix}.
\]
\end{lem}

We define a projection $\mathcal{P}_0$ from $\mathcal{H}$ to $\mathcal{H}_1$
by 
$$
\mathcal{P}_0[\Psi]:= (\Psi, \Psi_0)\Phi_0,
$$
and denote by
$$
\tilde{\mathcal{A}_0}= \mathcal{A}_0 + \mathcal{P}_0.
$$
The following results hold. 
\begin{lem} We have
\begin{enumerate}
\item[(i)]
The operator $\tilde{\mathcal{A}_0}$ is a bijective operator in
$\mathcal{L}(\mathcal{H}, \mathcal{H}_1)$. Moreover, 
$\tilde{\mathcal{A}_0}[\Psi_0]= \Phi_0 $;
\item[(ii)]
The adjoint of $\tilde{\mathcal{A}_0}$, $\tilde{\mathcal{A}_0}^*$, is a bijective operator in $\mathcal{L}(\mathcal{H}, \mathcal{H}_1)$. Moreover, 
$\tilde{\mathcal{A}_0}^*[\Phi_0] = \Psi_0$. 
\end{enumerate}
\end{lem}

\begin{proof}
By construction, and the fact that $\mathcal{S}_D$ is bijective from $L^2(\p D)$ to $H^1(\p D)$ \cite{book2}, we can show that $\tilde{\mathcal{A}_0}$ is a bijective. So too is $\tilde{\mathcal{A}_0}^*$. We only need to show that 
$\tilde{\mathcal{A}_0}^*[\Phi_0]= \Psi_0$.
Indeed, we can check that $\mathcal{P}_0^*[\theta]= (\theta, \Phi_0)\Psi_0$. Thus, it follows that 
$$
\tilde{\mathcal{A}_0}^* [\Phi_0] = \mathcal{P}_0^*[\Phi_0] = (\Phi_0, \Phi_0)\Psi_0
=\Psi_0,
$$
which completes the proof. 
\end{proof}

Our main result in this section is stated in the following theorem. 
\begin{thm} \label{thm-resoance}
In the  quasi-static regime, there exists two resonances for a single bubble: 
\beas
\omega_{0,0}(\delta) &=& \sqrt{ \f{{Cap}(D)}{\tau^2 v^2 Vol(D)}} \delta^{\f{1}{2}} -i\frac{{Cap(D)}^2}{8\pi \tau^2 v Vol(D) } \delta +O(\delta^{\f{3}{2}}),\\
\omega_{0,1}(\delta) &=& -\sqrt{ \f{Cap(D)}{\tau^2 v^2 Vol(D)}} \delta^{\f{1}{2}}  -i\frac{ Cap(D)^2}{8\pi \tau^2 v Vol(D) } \delta +O(\delta^{\f{3}{2}}),
\eeas 
where $Vol(D)$ is the volume of $D$ and $Cap(D):= - (\psi_0, \chi_{\p D}) = - (\mathcal{S}_{D}^{-1}[\chi_{\p D}], \chi_{\p D})$ is the capacity of $D$. The first resonance $\omega_{0,0}$ is called the Minnaert resonance. 
\end{thm}

\begin{proof}
Step 1.  We find the resonance by solving the following equation
\be \label{eq-resonance-1}
\mathcal{A}(\omega, \delta)[\Psi_{\delta}] =0. 
\ee
Write $\Psi_{\delta} = \Psi_0 + \Psi_1$ and assume the orthogonality condition
\be \label{orthog}
(\Psi_1, \Psi_0)=0. 
\ee

Step 2. Since $\tilde{\mathcal{A}_0}= \mathcal{A}_0 + \mathcal{P}_0$, (\ref{eq-resonance-1}) is equivalent to the following 
$$
(\tilde{\mathcal{A}_0} - \mathcal{P}_0 + \mathcal{B}) [\Psi_0 + \Psi_1]=0.
$$
Observe that as the operator $\tilde{\mathcal{A}_0} + \mathcal{B}$ is invertible for sufficiently small $\delta$ and $\omega$, we can apply $(\tilde{\mathcal{A}_0} + \mathcal{B})^{-1}$ to both sides of the above equation to deduce that
\be  \label{eq-resonance-2}
\Psi_1= (\tilde{\mathcal{A}_0} + \mathcal{B})^{-1} \mathcal{P}_0 [\Psi_0] - \Psi_0
= (\tilde{\mathcal{A}_0} + \mathcal{B})^{-1}[\Phi_0] - \Psi_0.
\ee

Step 3. Using the orthogonality condition (\ref{orthog}), we arrive at the following equation: 
\be  \label{eq-algebraic}
A(\omega, \delta):= \left((\tilde{\mathcal{A}_0} + \mathcal{B})^{-1} [\Phi_0], \Psi_0\right) - 1=0
\ee

Step 4. We calculate $A(\omega, \delta)$. 
Using the identity
$$
(\tilde{\mathcal{A}_0} + \mathcal{B})^{-1} 
= \left(Id+ \tilde{\mathcal{A}_0}^{-1} \mathcal{B}\right)^{-1} \tilde{\mathcal{A}_0}^{-1}
= \left( Id- \tilde{\mathcal{A}_0}^{-1}\mathcal{B} + \tilde{\mathcal{A}_0}^{-1}\mathcal{B}\tilde{\mathcal{A}_0}^{-1}\mathcal{B}+...\right)\tilde{\mathcal{A}_0}^{-1}, 
$$
and the fact that 
$$
\tilde{\mathcal{A}_0}^{-1}[\Phi_0] = \Psi_0,
$$
we obtain
\beas
A(\omega, \delta)
&=& -\omega \left( \mathcal{A}_{1,0}[\Psi_0], \Phi_0\right)
-\omega^2 \left( \mathcal{A}_{2,0}[\Psi_0], \Phi_0\right)
-\omega^3 \left( \mathcal{A}_{3,0}[\Psi_0], \Phi_0\right)  -\delta \left( \mathcal{A}_{0,1}[\Psi_0], \Phi_0\right)\\
&& + \omega^2 \left( \mathcal{A}_{1,0}\tilde{\mathcal{A}_0}^{-1}\mathcal{A}_{1,0}[\Psi_0], \Phi_0\right) 
+ \omega^3 \left( \mathcal{A}_{1,0}\tilde{\mathcal{A}_0}^{-1}\mathcal{A}_{2,0}[\Psi_0], \Phi_0\right) 
+\omega^3 \left( \mathcal{A}_{2,0}\tilde{\mathcal{A}_0}^{-1}\mathcal{A}_{1,0}[\Psi_0], \Phi_0\right) \\
&& + \omega \delta \left( \mathcal{A}_{1,0}\tilde{\mathcal{A}_0}^{-1}\mathcal{A}_{0,1}[\Psi_0], \Phi_0\right)  
 + \omega \delta \left( \mathcal{A}_{0,1}\tilde{\mathcal{A}_0}^{-1}\mathcal{A}_{1,0}[\Psi_0], \Phi_0\right) \\
&& + \omega^3 \left( \mathcal{A}_{1,0}\tilde{\mathcal{A}_0}^{-1}\mathcal{A}_{1,0}\tilde{\mathcal{A}_0}^{-1}\mathcal{A}_{1,0}[\Psi_0], \Phi_0\right) + O(\omega^4)+ O(\delta^2).  
\eeas
%

It is clear that $\mathcal{A}_{1,0}^*[\Phi_0]=0$. Consequently, we get
\beas
A(\omega, \delta)&=& 
-\omega^2 \left( \mathcal{A}_{2,0}[\Psi_0], \Phi_0\right)-\omega^3 \left( \mathcal{A}_{3,0}[\Psi_0], \Phi_0\right) -\delta \left( \mathcal{A}_{0,1}[\Psi_0], \Phi_0\right) 
\\
&& + \omega^3 \left( \mathcal{A}_{2,0}\tilde{\mathcal{A}_0}^{-1}\mathcal{A}_{1,0}[\Psi_0], \Phi_0\right)+ \omega \delta \left( \mathcal{A}_{0,1}\tilde{\mathcal{A}_0}^{-1}\mathcal{A}_{1,0}[\Psi_0], \Phi_0\right) + O(\omega^4)+ O(\delta^2).
\eeas

In the next four steps, we calculate the terms
$\left( \mathcal{A}_{2,0}[\Psi_0], \Phi_0\right)$,  $\left( \mathcal{A}_{3,0}[\Psi_0], \Phi_0\right)$, $\left( \mathcal{A}_{0,1}[\Psi_0], \Phi_0\right)$,
$\left( \mathcal{A}_{2,0}\tilde{\mathcal{A}_0}^{-1}\mathcal{A}_{1,0}[\Psi_0], \Phi_0\right)$ and 
$\left( \mathcal{A}_{0,1}\tilde{\mathcal{A}_0}^{-1}\mathcal{A}_{1,0}[\Psi_0], \Phi_0\right)$. 

Step 5. We have
\beas
\left( \mathcal{A}_{2,0}[\Psi_0], \Phi_0\right)&=& \alpha_0\beta_0 \tau^2v^2 
\left( \mathcal{K}_{D, 2}[\psi_0], \phi_0 \right)= \alpha_0\beta_0 \tau^2v^2 
\left( \psi_0, \mathcal{K}_{D, 2}^*[\phi_0] \right) \\ 
&=& -\alpha_0\beta_0 \tau^2v^2 
\int_{\p D} d\sigma(x)\mathcal{S}_{D}^{-1}[\chi_{\p D}](x) \int_{D}dy G(x, y,0) \chi_{\p D}(y) \\
&=& -\alpha_0\beta_0 \tau^2v^2  \int_{D}dy  \int_{\p D} d\sigma(x) G(x, y, 0)\mathcal{S}_{D}^{-1}[\chi_{\p D}](x)  \\
&=& -\alpha_0\beta_0 \tau^2v^2 \int_{D} \chi(y)dy \\
&=& -\alpha_0\beta_0 \tau^2v^2 Vol (D).  
\eeas

Step 6. On the other hand, we have
\beas
\left( \mathcal{A}_{3,0}[\Psi_0], \Phi_0\right)&=& \alpha_0\beta_0 \tau^3v^3 
\left( \psi_0, \mathcal{K}_{D, 3}^*[\phi_0] \right)
=\alpha_0\beta_0 \tau^3v^3\left( \psi_0,  \f{i}{4\pi} Vol (D) \right)\\
& = & \alpha_0\beta_0 \tau^3v^3 Vol (D) \f{i}{4\pi}\left( \mathcal{S}_{D}^{-1}[\chi_{\p D}],  \chi_{\p D} \right) = - \alpha_0\beta_0 \tau^3v^3 Vol (D) \f{i}{4\pi} Cap(D).
\eeas

Step 7. It is easy to see that
$$
\left( \mathcal{A}_{0,1}[\Psi_0], \Phi_0\right)
= -(\psi_0, \phi_0) = -\alpha_0\beta_0 \left(\mathcal{S}_{D}^{-1} [\chi_{\p D}],  \chi_{\p D} \right) =  \alpha_0\beta_0 Cap(D).
$$

Step 8. We now calculate the term $\left( \mathcal{A}_{0,1}\tilde{\mathcal{A}_0}^{-1}\mathcal{A}_{1,0}[\Psi_0], \Phi_0\right) $.
We have
\beas
\mathcal{A}_{1, 0}[\Psi_0]&=&
\begin{pmatrix}
   (\tau-1)v \mathcal{S}_{D,1}[\psi_0]\\
  0
\end{pmatrix}
=\begin{pmatrix}
  (\tau-1)v \f{i}{4 \pi} Cap(D) \chi_{\p D} \\
   0
\end{pmatrix},\\
\mathcal{A}_{0, 1}^*[\Phi_0]&=&
\begin{pmatrix}
  0\\
  -\left(\f{1}{2}Id+ K_D \right)[\phi_0]
\end{pmatrix}
=\begin{pmatrix}
  0\\
  -\phi_0
\end{pmatrix}
=-\begin{pmatrix}
  0\\
  \chi_{\p D}
\end{pmatrix} . 
\eeas

We need to calculate 
\[ 
\tilde{\mathcal{A}}_0^{-1} \begin{pmatrix}
  \chi_{\p D}\\
  0
\end{pmatrix}. 
\]
Assume that 
\[
\left( \mathcal{A}_0 + \mathcal{P}_0 \right) 
\begin{pmatrix}
  y_b\\
  y
\end{pmatrix}
= \begin{pmatrix}
  \mathcal{S}_D [y_b-y]\\
  (-\f{1}{2}Id + \mathcal{K}_D^*)[y_b] 
  \end{pmatrix}
  + \left( (y_b, \psi_0)+ (y, \psi_0)    \right)\begin{pmatrix}
  0\\
  \phi_0
\end{pmatrix}
=\begin{pmatrix}
  \chi(\p D)\\
  0
\end{pmatrix}
\]
By solving the above equations directly, we obtain that
$ y_b= \f{1}{2} \psi_0, y= -\f{1}{2} \psi_0$. Therefore, 
\[
\tilde{\mathcal{A}}_0^{-1} \begin{pmatrix}
  \chi_{\p D}\\
  0
\end{pmatrix} = 
\begin{pmatrix}
\f{1}{2} \psi_0\\
-\f{1}{2} \psi_0
\end{pmatrix}.
\]
It follows that
\[
\left( \mathcal{A}_{0,1}\tilde{\mathcal{A}_0}^{-1}\mathcal{A}_{1,0}[\Psi_0], \Phi_0\right) = (\tau -1)v \f{i}{8 \pi} Cap(D) (\psi_0, \phi_0)
= (1- \tau) v \f{i}{8 \pi} Cap(D)^2\alpha_0\beta_0.
\]

Step 9. We calculate the term $\left( \mathcal{A}_{2,0}\tilde{\mathcal{A}_0}^{-1}\mathcal{A}_{1,0}[\Psi_0], \Phi_0\right) $. Using the results in Step 8, we obtain
\beas
\left( \mathcal{A}_{2,0}\tilde{\mathcal{A}_0}^{-1}\mathcal{A}_{1,0}[\Psi_0], \Phi_0\right) 
&=& \left( \tilde{\mathcal{A}_0}^{-1}\mathcal{A}_{1,0}[\Psi_0], \mathcal{A}_{2,0}^* [\Phi_0]\right) \\
&=&  \f{i(\tau-1)\tau^2v^3}{8 \pi} Cap(D) \alpha_0 \beta_0 \, \big(\psi_0, \mathcal{K}_{D, 2}^*[\phi_0] \big)\\
&=& \f{i(1 -\tau)\tau^2v^3}{8 \pi} Cap(D) Vol(D) \alpha_0 \beta_0.
\eeas

Step 10. Considering the above the results, we can derive
\beas
A(\omega, \delta) &= &\alpha_0\beta_0 \left( 
\tau^2 v^2 Vol(D) \omega^2 + \f{i\tau^2(\tau+1) v^3 Vol(D) Cap(D)}{8\pi} \omega^3 
 - Cap(D) \delta - \f{i(\tau-1)v Cap(D)^2}{8 \pi} \omega \delta 
 \right) \\
 && + O(\omega^4) + O(\delta^2).
\eeas

We now solve $A(\omega, \delta) =0$. 
It is clear that $\delta = O(\omega^2)$, and thus $\omega_0(\delta) = O(\sqrt{\delta})$. 
Write 
$$
\omega_0(\delta) = a_1 \delta^{\f{1}{2}} + a_2 \delta + O (\delta^{\f{3}{2}}). 
$$
We get
\beas
&&\tau^2 v^2 Vol(D) \left(a_1 \delta^{\f{1}{2}} + a_2 \delta + O (\delta^{\f{3}{2}})\right)^2 + \f{i\tau^2(\tau+1) v^3 Vol(D) Cap(D)}{8\pi} \left(a_1 \delta^{\f{1}{2}} + a_2 \delta + O (\delta^{\f{3}{2}})\right)^3 \\
&& - Cap(D) \delta - \f{i(\tau-1)v Cap(D)^2}{8 \pi} \left(a_1 \delta^{\f{1}{2}} + a_2 \delta + O (\delta^{\f{3}{2}})\right) \delta
+O(\delta^2)=0.
\eeas
From the coefficients of the $\delta$ and $\delta^{\f{3}{2}}$ terms, we obtain 
\beas
&& \tau^2 v^2 Vol(D)a_1^2 - Cap(D) =0, \\
&& 2\tau^2 v^2 Vol(D)a_1a_2 + \f{i\tau^2(\tau+1) v^3 Vol(D) Cap(D)}{8\pi}a_1^3
- \f{i(\tau-1)v Cap(D)^2}{8 \pi} a_1=0,
\eeas 
which yields 
 \beas
a_1 &=& \pm \sqrt{ \f{Cap(D)}{\tau^2 v^2 Vol(D)}},\\
a_2&=& -\f{i(\tau+1) v Cap(D)}{16\pi}a_1^2 + \f{i(\tau-1) Cap(D)^2}{16 \pi \tau^2 v Vol(D)}
= -\frac{i (\tau+1) Cap(D)^2}{16\pi \tau^2 v Vol(D) }+ \f{i(\tau-1) Cap(D)^2}{16 \pi \tau^2 v Vol(D)}\\
&=& \frac{-i Cap(D)^2}{8\pi \tau^2 v Vol(D) }.
\eeas 
This complete the proof of the theorem.

\end{proof}

A few remarks are in order.

\begin{rmk}
Using the method developed above, we can derive the Minnaert resonance for a single bubble in two dimensions. The main differences between the two-dimensional case and the three-dimensional case are explained in Appendix \ref{appendixB}. 
\end{rmk}

\begin{rmk}
Using the method developed above, we can also obtain the full asymptotic expansion for the resonance with respect to the small parameter $\delta$.
\end{rmk}

\begin{rmk}
In the case of a collection of $N$ identical bubbles, with separation distance much larger than their characteristic sizes, the Minnaert resonance for a single bubble will be split into N resonances. The splitting will be related to the eigenvalues of a N-by-N matrix which encodes information on the configuration of the N bubbles. This can be proved by a similar argument as in \cite{hai}.
\end{rmk}

\begin{rmk}
Taking into consideration the above theorem, we can deduce that if the bubble is represented by $D= t B$ for some small positive number $t$ and a normalized domain $B$ with size of order one, then the Minnaert resonance for $D$ is given by
the following formula
\beas
\omega_{0,0}(\delta) = \f{1}{t} \left[\sqrt{ \f{Cap(B)}{\tau^2 v^2 Vol(B)}} \delta^{\f{1}{2}} -i\frac{ Cap(B)^2}{8\pi \tau^2 v Vol(B) } \delta +O(\delta^{\f{3}{2}}) \right]. 
\eeas 

\end{rmk}

\begin{rmk}
In the special case when $D$ is the unit sphere, we have $Cap(D) = 4\pi$,
$Vol(D)= \frac{4 \pi}{3}$. Consequently,  
\beas
\sqrt{ \f{Cap(D)}{\tau^2 v^2 Vol(D)}}&=& 
 \sqrt {3} \frac{1}{v_b}, \\
\frac{ Cap(D)^2}{8\pi \tau^2 v Vol(D) }&=& \f{3}{2 \tau^2 v}. 
\eeas
Therefore, the Minnaert resonance is given by 
\beas
\omega_{0,0} (\delta) &=&  \sqrt {3} \frac{1}{v_b} \delta^{\f{1}{2}} 
-i \f{3}{2 \tau^2 v} \delta + O(\delta^{\f{3}{2}}),\\
&=& \sqrt{\f{3\kappa_b}{\rho}} -i \f{3}{2} \kappa_b \sqrt{\f{1}{\rho \kappa}} +O((\f{\rho_b}{\rho})^{\f{3}{2}}).  
\eeas
\end{rmk}


\section{The point scatterer approximation} \label{sec-3d-point-scatter}
We now solve the scattering problem (\ref{eq-scattering})
with $u^{in}= e^{ikd\cdot x}$. This models the case when the bubble is excited by sources in the far field (throughout the paper, a point $x$ is said to be in the far field of the bubble $D$ if the distance between $x$ and $D$ is much larger than the size of $D$). 
The problem is equivalent to equation (\ref{eq-boundary})
with $F$ being determined by 
\[
F= 
\begin{pmatrix}
u^{in}\\
\delta \f{\partial u^{in}}{\partial \nu}
\end{pmatrix}.
\]

We need the following lemma. 
\begin{lem} \label{lem-estimate1}
The following estimates hold in $\mathcal{H}$:
\[
(\tilde{\mathcal{A}_0}+ \mathcal{B})^{-1}[F] = u^{in}(y_0)
\begin{pmatrix}
\f{1}{2} \psi_0\\
-\f{1}{2} \psi_0
\end{pmatrix}+ 
O(\omega)+ O(\delta).
\]
\end{lem}

\begin{proof}
Let $
F= F_1+ F_2,
$
where 
\[
F_1= 
\begin{pmatrix}
u^{in}(y_0) \chi_{\p D}\\
0
\end{pmatrix}
, \quad F_2= F-F_1 =
\begin{pmatrix}
O(\omega)\\
\delta \f{\partial u^{in}}{\partial \nu}
\end{pmatrix}.
\]
It is clear that $F_2 = O(\omega)$ in $\mathcal{H}_1$. Using the fact that 
$$
(\tilde{\mathcal{A}_0}+ \mathcal{B})^{-1} = \tilde{\mathcal{A}_0}^{-1} + O(\omega) + O(\delta),
$$
we obtain 
\beas
(\tilde{\mathcal{A}_0}+ \mathcal{B})^{-1}[F] &=&
 (\tilde{\mathcal{A}_0}+ \mathcal{B})^{-1}[F_1] + (\tilde{\mathcal{A}_0}+ \mathcal{B})^{-1}[F_2], \\
 &=& \tilde{\mathcal{A}_0}^{-1}[F_1] + O(\omega)+ O(\delta), \\
&=& u^{in}(y_0)
\begin{pmatrix}
\f{1}{2} \psi_0\\
-\f{1}{2} \psi_0
\end{pmatrix}+ 
O(\omega) + O(\delta),
\eeas
which is the desired result. 
\end{proof}

%
%
%
%

The following monopole approximation holds. 
\begin{thm} \label{thm2}
In the far field, the solution to the scattering problem (\ref{eq-scattering})
has the following point-wise behavior 
\beas
u^s(x)= g(\omega, \delta, D)\left( 1+ O(\omega) + O(\delta) + o(1)\right) u^{in}(y_0) 
G(x, y_0, k), 
\eeas
where $y_0$ is the center of the bubble and the scattering coefficient $g$ is given below: 
\begin{enumerate}
\item[(i)]
Regime I:  $\omega \ll \sqrt{\delta}$,
\be  \label{g-3}
g(\omega, \delta, D) =O(\f{\omega^2}{\delta})+O(\omega);
\ee
\item[(ii)]
Regime II:  $\f{\omega}{\sqrt{\delta}} =O(1)$, 
\be  \label{g-2}
g(\omega, \delta, D) =  
\f{Cap(D)} {1- (\f{\omega_M}{\omega})^2 + i\gamma},
\ee 
where 
\[
\omega_M= \sqrt{\f{Cap(D) \delta}{\tau^2 v^2 Vol(D)} },\quad
\gamma = \f{(\tau +1)v Cap(D)\omega}{8 \pi} - \f{(\tau-1)Cap(D)^2 \delta}{8 \pi \tau^2 v Vol(D) \omega}
\]
are called the Minnaert resonance frequency and the damping constant respectively.  
In particular, the  Minnaert resonance occurs  in this regime.
\item[(iii)]
Regime III: $\sqrt{\delta} \ll \omega \ll 1$, 

\be  \label{g-1}
g(\omega, \delta, D) = Cap(D) + O(\frac{\delta}{\omega}).
\ee
\end{enumerate}
\end{thm}
\begin{proof}
Step 1. We write $\Psi= \alpha u^{in}(y_0) \Psi_0 + \Psi_1$ with $(\Psi_1, \Psi_0)=0$. Then,
\beas
&& (\tilde{\mathcal{A}_0} - \mathcal{P}_0 + \mathcal{B}) [\alpha u^{in}(y_0)\Psi_0 + \Psi_1] =F
\eeas
implies that
\beas
&& \left(Id- (\tilde{\mathcal{A}_0}+ \mathcal{B})^{-1}\mathcal{P}_0 \right) [\alpha u^{in}(y_0) \Psi_0 + \Psi_1] = (\tilde{\mathcal{A}_0}+ \mathcal{B})^{-1}[F], 
\eeas
which yields
\beas
&& \alpha u^{in}(y_0) \Psi_0 + \Psi_1 
- \alpha u^{in}(y_0) (\tilde{\mathcal{A}_0}+ \mathcal{B})^{-1}\Phi_0 = 
(\tilde{\mathcal{A}_0}+ \mathcal{B})^{-1}[F].
\eeas
As a result, we get
\beas
\alpha u^{in}(y_0) &=& \frac{((\tilde{\mathcal{A}_0}+ \mathcal{B})^{-1}[F], \Psi_0 )}{1-\left( (\tilde{\mathcal{A}_0}+ \mathcal{B})^{-1}[\Phi_0], \Psi_0 \right)}
= - \frac{((\tilde{\mathcal{A}_0}+ \mathcal{B})^{-1}[F], \Psi_0 )} {A(\omega, \delta)}, \\
\Psi_1&=& (\tilde{\mathcal{A}_0}+ \mathcal{B})^{-1}[F] + 
\alpha u^{in}(y_0) (\tilde{\mathcal{A}_0}+ \mathcal{B})^{-1}[\Phi_0] - \alpha u^{in}(y_0) \Psi_0.
\eeas
By Lemma \ref{lem-estimate1}, we have
\[
\Psi_1 =u^{in}(y_0)
\begin{pmatrix}
\f{1}{2} \psi_0\\
-\f{1}{2} \psi_0
\end{pmatrix}+ 
O(\omega)+ O(\delta).
 \]
 
Step 2. We calculate the scattered far field. Note that
\beas
\mathcal{S}_{D}^k [\psi_0] (x) &=& \int_{\p D} G(x, y, k) \psi_0(y) d\sigma(y)
= \int_{\p D} G(x, y_0, k)(1+O(\omega) + o(1)) \psi_0(y) d\sigma(y) \\
\nm
&= & G(x, y_0, k) (\chi_{\p D}, \mathcal{S}_D^{-1}[\chi_D])(1+ O(\omega)+ o(1)) \\
\nm
&=& - Cap(D) G(x, y_0, k)(1+ O(\omega)+ o(1)).
\eeas
Therefore,
\beas
u^s(x ) &=& ( \alpha_0 \alpha u^{in}(y_0) -\f{1}{2} u^{in}(y_0) + O(\omega) + O(\delta)) 
\mathcal{S}_{D}^k (\psi_0)(x)  \\
\nm
&=&  -( \alpha_0 \alpha u^{in}(y_0) -\f{1}{2} u^{in}(y_0) + O(\omega) + O(\delta)) Cap(D) G(x, y_0, k)
(1+ O(\omega)+ o(1)) ,\\
\nm
&=& g(\omega, \delta, D) u^{in}(y_0) G(x, y_0, k) (1+ O(\omega)+  O(\delta) +o(1)),  
\eeas
where we have introduced
\be
g(\omega, \delta, D)= - (\alpha_0 \alpha -\f{1}{2} ) Cap(D). 
\ee
$g$ is called the scattering coefficient of the bubble.

Step 3. 
We prove that
\be \label{lem-estimate2}
\alpha
=\f{\left[\omega^2 \tau^2v^2 Vol(D)+ \delta Cap(D)\right]  \beta_0 + O(\delta \omega) +O(\omega^3)}{-2A(\omega, \delta)}.
\ee
Let 
$
F= F_1+ F_2 ,
$
where 
\[
F_1= 
\begin{pmatrix}
u^{in}\\
0
\end{pmatrix}
, \quad F_2= F-F_1 =
\begin{pmatrix}
0\\
\delta \f{\partial u^{in}}{\partial \nu}
\end{pmatrix}.
\]
Then 
\[
\alpha u^{in}(y_0)= - \frac{\big( (\tilde{\mathcal{A}_0}+ \mathcal{B})^{-1} [F_1], \Psi_0 \big) + \big(( 
\tilde{\mathcal{A}_0}+ \mathcal{B})^{-1}[F_2],  \Psi_0 \big)} {A(\omega, \delta)}:= - \f{I_1+I_2}{A(\omega, \delta)}.
\]
It is clear that $F_2 = O(\delta \omega)$ in $\mathcal{H}_1$, and thus
$$
I_2= ((\tilde{\mathcal{A}_0}+ \mathcal{B})^{-1}[F_2], \Psi_0 )=O(\delta \omega). 
$$
We now investigate $I_1= ((\tilde{\mathcal{A}_0}+ \mathcal{B})^{-1} [F_1], \Psi_0 )$. 
We have
\beas
I_1
&=&  \left( (Id- \tilde{\mathcal{A}_0}^{-1}\mathcal{B} + \tilde{\mathcal{A}_0}^{-1}\mathcal{B}\tilde{\mathcal{A}_0}^{-1}\mathcal{B}+...)\tilde{\mathcal{A}_0}^{-1}[F_1], \Psi_0 \right) \\
&=& (\tilde{\mathcal{A}_0}^{-1}[F_1], \Psi_0) - (\mathcal{B}\tilde{\mathcal{A}_0}^{-1}[F_1], \Phi_0) + (\mathcal{B}\tilde{\mathcal{A}_0}^{-1}\mathcal{B} \tilde{\mathcal{A}_0}^{-1}[F_1], \Phi_0) +... \\
&=& (F_1, \Phi_0)- (\tilde{\mathcal{A}_0}^{-1}[F_1], \mathcal{B}^*\Phi_0) + (\tilde{\mathcal{A}_0}^{-1}\mathcal{B} \tilde{\mathcal{A}_0}^{-1}[F_1],\mathcal{B}^*[\Phi_0]) + ... \\
&=&-(\tilde{\mathcal{A}_0}^{-1}[F_1], \mathcal{B}^*[\Phi_0]) + (\tilde{\mathcal{A}_0}^{-1}\mathcal{B} \tilde{\mathcal{A}_0}^{-1}[F_1],\mathcal{B}^*[\Phi_0]) + ...,
\eeas
where we have used the fact that $(F_1, \Phi_0)=0$ and $(\tilde{\mathcal{A}_0}^{-1})^*[\Psi_0]= \Phi_0$.

Note that
\[
\mathcal{B}^*[\Phi_0] = \omega \mathcal{A}_{1,0}^*[\Phi_0]
+ \omega^2  \mathcal{A}_{2,0}^*[\Phi_0]+ \omega^3 \mathcal{A}_{3,0}^*[\Phi_0]
+ \delta \mathcal{A}_{0,1}^*[\Phi_0]
+ O(\omega^4)+ O(\delta \omega^2). 
\]
Using the facts that
\[
\tilde{\mathcal{A}_0}^{-1}[F_1] = u^{in}(y_0)
\begin{pmatrix}
\f{1}{2} \psi_0\\
-\f{1}{2} \psi_0
\end{pmatrix}+ 
O(\omega),
\]
and 
\beas
\mathcal{A}_{1, 0}^*[\Phi_0]&=
 0, \quad
\mathcal{A}_{2, 0}^*[\Phi_0]=
\beta_0 \begin{pmatrix}
  \tau^2 v^2 \mathcal{K}_{D,2}^*[\phi_0]\\
  0
\end{pmatrix}, \\
\mathcal{A}_{3, 0}^*[\Phi_0]&=
\beta_0 \begin{pmatrix}
  \tau^3 v^3 \mathcal{K}_{D,3}^*[\phi_0]\\
  0
\end{pmatrix}, \quad
\mathcal{A}_{0, 1}^*[\Phi_0]=
-\beta_0\begin{pmatrix}
  0\\
  \chi_{\p D}
\end{pmatrix},
\eeas
we can conclude that
\beas
I_1&=& -\left(\tilde{\mathcal{A}_0}^{-1}[F_1], 
 \omega^2  \mathcal{A}_{2,0}^*[\Phi_0]+ \delta \mathcal{A}_{0,1}^* [\Phi_0] + O(\delta \omega)+ O(\omega^3) \right),\\
&=& -\f{1}{2}u^{in}(y_0) \beta_0\left[ \omega^2 (\psi_0, \tau^2 v^2 \mathcal{K}_{D,2}^*[\phi_0])
+ \delta (\psi_0, \chi_{\p D}) \right] + O(\delta \omega)+ O(\omega^3),\\
&=&\f{1}{2}\left(\omega^2 \tau^2v^2 Vol(D)+ \delta Cap(D)\right)  \beta_0 u^{in}(y_0) + O(\delta \omega) +O(\omega^3),
\eeas
which completes the proof of (\ref{lem-estimate2}).

Step 4. Recall the formula for $A(\omega, \delta)$ in the previous section and (\ref{lem-estimate2}), we have
\beas
- \f{2g(\omega, \delta, D)}{Cap(D)} =
 \f{ -\omega^2 \tau^2v^2 Vol(D)- \delta Cap(D) +O(\delta  \omega) +O(\omega^3)}
 {\tau^2 v^2 Vol(D) \omega^2 +\f{i\tau^2(\tau+1) v^3 Vol(D) Cap(D)}{8\pi} \omega^3 
 - Cap(D) \delta - \f{i(\tau-1)v Cap(D)^2}{8 \pi} \omega \delta
+ O(\omega^4) + O(\delta^2)}-1. 
 \eeas
The asymptotic behavior of $g$ in different regimes follows immediately from the above formula. This completes the proof of the theorem. 
\end{proof}


\begin{rmk}
Using the method developed above together with the results of Appendix \ref{appendixB}, we can derive  a similar monopole approximation in the far field  for a single bubble in two dimensions. 
\end{rmk}

\section{Numerical illustrations} \label{sec:numerics}
In this section we perform numerical simulations in two dimensions to analyze the resonant frequencies for two scenarios. We first analyze the single bubble case for which a formula was derived in Theorem \ref{thm-resonance-2d}. We then calculate the resonant frequencies for two bubbles and compare our results with the single bubble case.

\subsection{Resonant frequency of a single bubble in two dimensions}
To validate the Minnaert resonance formula (\ref{eq-resonance-2d-unit-circle}) in two dimensions we first determine the characteristic value $\omega_c$ of $\mathcal{A}(\omega, \delta)$ in (\ref{eq-resonance}) numerically. We then calculate the complex root $\omega_f$ of (\ref{eq-resonance-2d-unit-circle}) that has a positive real part. Comparing $\omega_c$ and $\omega_r$ over a range of appropriate values of $\delta$ allows us to judge the accuracy of the formula.

In order to perform the analysis in the correct regime, which was described in Section \ref{sec-minnaert}, we take $\rho = \kappa = 1000$ and $\rho_b = \kappa_b = c$, where $c$ is chosen such that the wave speed in both air and water is of order 1 and $\delta \in \{10^{-i}\}, \ i \in \{1,\dots, 5\}$. We use $2^9$ points to discretize the unit circle used in the calculation of the layer potentials that form $\mathcal{A}$. Calculating $\omega_c$ is equivalent to determining the smallest $\omega$ such that $\mathcal{A}(\omega, \delta)$ has a zero eigenvalue. We have 
$$
\omega_c = \min\limits_{\omega \in \C}\{\omega| \ \lambda(\omega) = 0\}\quad \lambda \in \sigma{(\mathcal{A}(\omega, \delta))},
$$
and we approach $\lambda(\omega) = 0$ as a complex root finding problem which can be calculated using Muller's method \cite{akl, cheng}.
Muller's method is applied again in order to find the root $\omega_f$ satisfying (\ref{eq-resonance-2d-unit-circle}). The resonant frequencies $\omega_c$ and $\omega_f$, along with the relative errors, for specific values of $\delta$ are given in Table \ref{table-results}. In Figure \ref{1-bubble_spec-A-vs-re-omega} it can be seen that the relative error becomes very small when $\delta \ll 1$, confirming the excellent accuracy of the formula. In particular, we note that when $\delta = 10^{-3}$, which is close to the usual contrast between water and air, the difference between $\omega_c$ and $\omega_f$ is negligible with a relative error of only $0.0652\%$.

\afterpage{
\begin{figure} 
\begin{center}
  \includegraphics[width=1\textwidth]{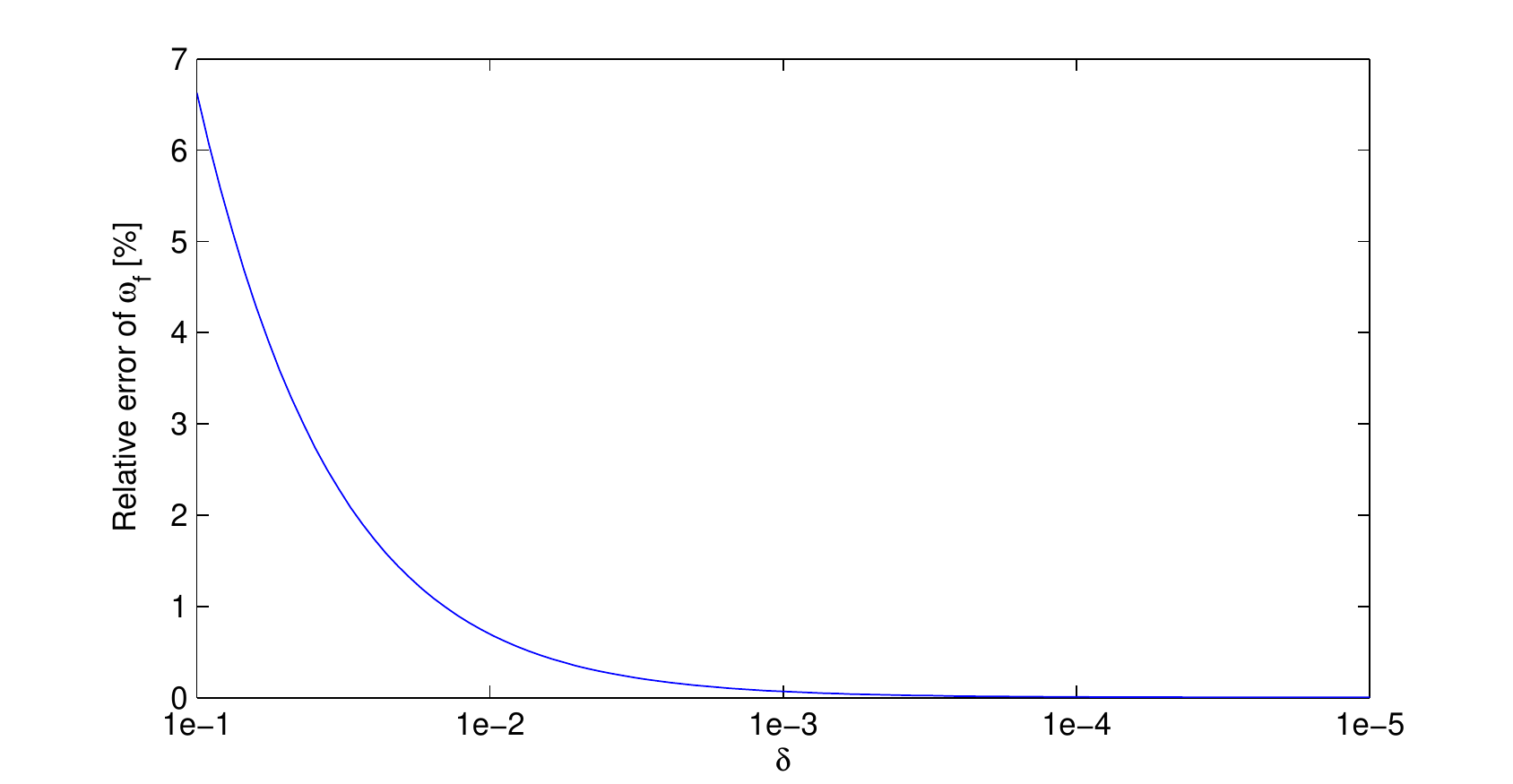}
  \caption{The relative error of the Minnaert resonance $\omega_c$ obtained by the two dimensional formula (\ref{eq-resonance-2d-unit-circle}) becomes negligible when we are in the appropriate high contrast regime.} \label{1-bubble_spec-A-vs-re-omega}
\end{center}
\end{figure}

\setlength{\tabcolsep}{12pt}
\begin{table}
    \centering
    \begin{tabular}{CCCC}
        \toprule
        \delta & \omega_c & \omega_f & \text{Relative error} \\        
        \midrule
        10^{-1} & 0.261145 - 0.150949i & 0.250455 - 0.134061i & 5.8203\% \\
        10^{-2} & 0.075146 - 0.023976i & 0.074681 - 0.023687i & 0.6727\% \\
        10^{-3} & 0.021001 - 0.004513i & 0.020987 - 0.004508i & 0.0652\% \\
        10^{-4} & 0.005950 - 0.000959i & 0.005949 - 0.000959i & 0.0062\% \\
        10^{-5} & 0.001714 - 0.000221i & 0.001714 - 0.000221i & 0.0030\% \\
        \bottomrule
    \end{tabular}
    \caption{A comparison between the characteristic value $\omega_c$ of $\mathcal{A}(\omega, \delta)$ and the root of the two dimensional resonance formula (\ref{eq-resonance-2d-unit-circle}) with positive real part $\omega_f$, over several values of $\delta$.}
    \label{table-results}
\end{table}
}

\subsection{Resonant frequencies of two bubbles in two dimensions}
In this subsection we numerically solve the two bubble case and analyze it with respect to our results for the Minnaert resonance of a single bubble. In the case of two bubbles we have two resonant frequencies, $\omega_s$ and $\omega_a$, that correspond to the normal modes of the system \cite{feuillade}. These frequencies are not in general equal to the one bubble resonant frequency $\omega_c$. The interaction between the bubbles gives rise to a shift in the resonance frequencies. The symmetric mode $\omega_s$ typically shows a downward frequency shift and occurs when the bubbles oscillate (collapse and expand) in phase, essentially opposing each other's motion. The antisymmetric mode $\omega_a$ shows an upward frequency shift and occurs when the bubbles oscillate in antiphase, facilitating each other's motion.

To account for the interaction between the two bubbles the matrix $\mathcal{A}$ in (\ref{eq-boundary}) is replaced with
\[
\mathcal{A}_2(\omega, \delta) = 
 \begin{pmatrix}
  \mathcal{S}_{D_{1}}^{k_b} &  -\mathcal{S}_{D_{1}}^{k} & 0 & -\mathcal{S}_{D_{1},D_{2}}^{k} \\
  -\f{1}{2}Id+ \mathcal{K}_{D_{1}}^{k_b, *} & -\delta( \f{1}{2}Id+ \mathcal{K}_{D_{1}}^{k, *}) & 0 & -\mathcal{K}_{D_{1},D_{2}}^{k, *} \\
  0 & -\mathcal{S}_{D_{2},D_{1}}^{k} & \mathcal{S}_{D_{2}}^{k_b} &  -\mathcal{S}_{D_{2}}^{k} \\
  0 & -\mathcal{K}_{D_{2},D_{1}}^{k, *} & -\f{1}{2}Id+ \mathcal{K}_{D_{2}}^{k_b, *}& -\delta( \f{1}{2}Id+ \mathcal{K}_{D_{2}}^{k, *}) \\
\end{pmatrix},
\]
where the operators $\mathcal{S}_{D_{ij}}^{k}$ and $\mathcal{K}_{D_{ij}}^{k_b, *}$ are given by
$$
\mathcal{S}_{D_{i},D_{j}}^{k} =  \int_{\p D_{j}} G(x, y, k) \psi(y) d\sigma(y),  \quad x \in  \p {D_{i}},
$$
and
$$
\mathcal{K}_{D_{i},D_{j}}^{k, *} [\psi](x)  = \int_{\p D_{j} } \f{\p G(x, y, k)}{ \p \nu(x)} \psi(y) d\sigma(y) ,  \quad x \in \p D_{i}.
$$

The variation in the eigenvalues of $\mathcal{A}_2$ with respect to the input frequency, and hence the shifting of the resonant frequencies, is highly sensitive to the ratio of $\delta = \rho_b/\rho$ to $\kappa_b/\kappa$, with it being at a minimum when these quantities are equal. In order to make the results more clearly visible, while keeping the simulation in the correct regime, let us take $\rho_b = 1.1$ and $\kappa_b = 0.1$. For reference, we note that the resonant frequency for a single bubble in this regime is $\omega_c = 0.01856427 - 0.00387243i$.

\begin{figure} [!ht]
\begin{center}
  \includegraphics[width=1\textwidth]{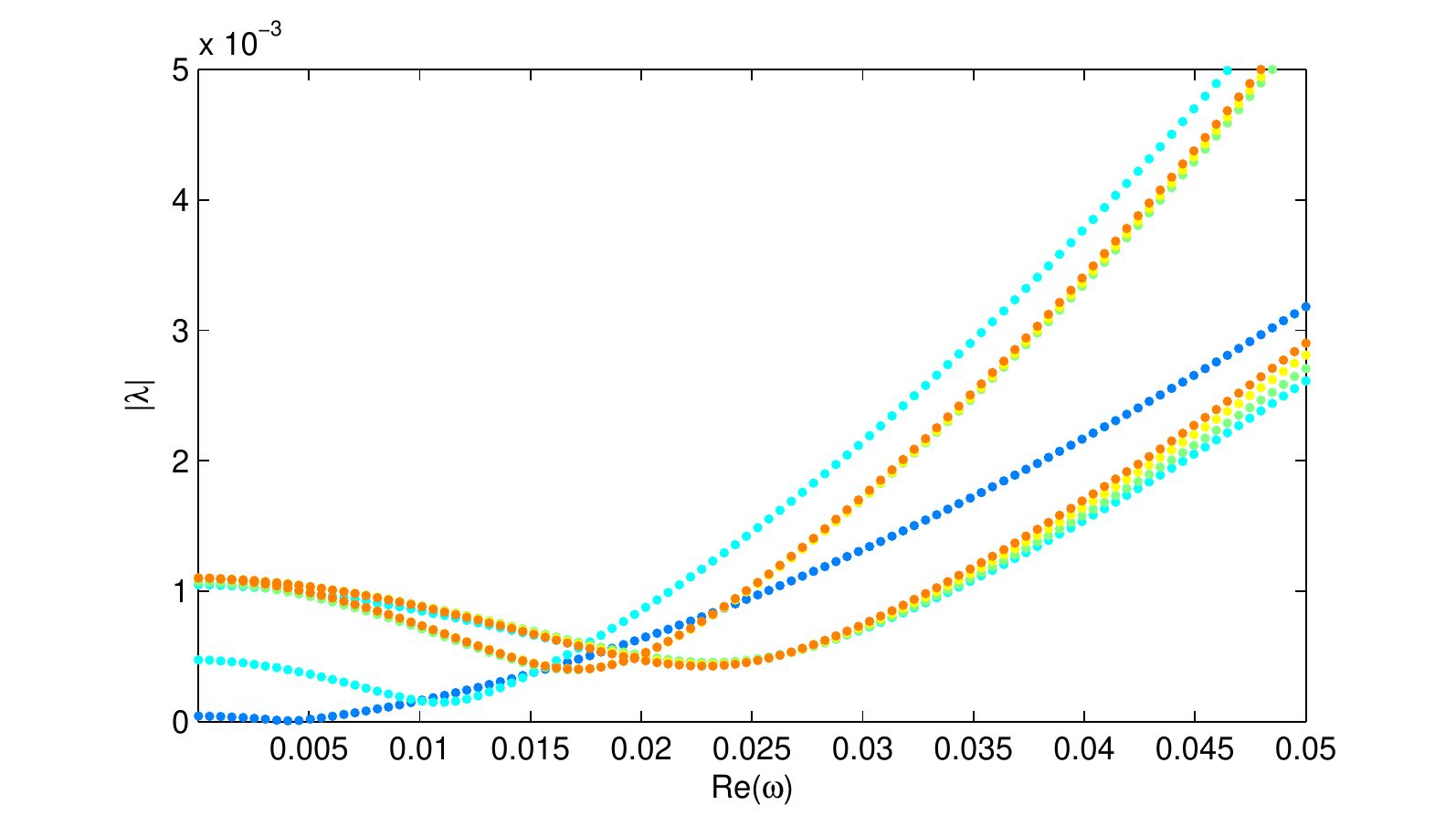}
  \caption{When the bubbles are close together the resonance may be much more pronounced. Here we have $|\lambda|$ as the distance varies from $d = 0.1$ (blue dots) to $d = 0.5$ (orange dots) and $\Im(\omega) = -0.0008i$. We have resonance at the symmetric mode $\omega_s \approx 0.0041 - 0.0008i$ when $d = 0.1$. The resonant frequency of a single bubble is $\omega_c = 0.01856427 - 0.00387243i$.}
  \label{fig-bubble-regime-close}
\end{center}
\end{figure}

We now identify three regimes in terms of bubble separation distance $d$. The first occurs due to strong interaction when $d$ is less than the radius of the bubbles. In this regime the resonant frequency shift may be much more pronounced. For example, when $d = 0.1$ we have $\omega_s \approx 0.0041 - 0.0008i$, while $\omega_a \approx 0.7435 + 0.0032i$. This regime is shown in Figure \ref{fig-bubble-regime-close} for $\Im(\omega) = - 0.008i$.

When $d$ is greater than the radius of the bubbles, yet not very large, we have a somewhat stable regime featuring small to moderate resonant frequency shifts. It is natural to expect that as the distance between the bubbles increases, the eigenvalues of the two bubble system approach those of the single bubble system. And indeed that is the case as can be seen in Figure \ref{fig-bubble-regime-mid-real} where $\omega$ has been restricted to $\R$. 

As with the three dimensional case, however, we require a complex $\omega$ with negative imaginary part in order for $\mathcal{A}$ or $\mathcal{A}_2$ to become singular. This can be seen in Figure \ref{fig-bubble-regime-mid-complex} for $d = 10$ and $d = 100$. Table \ref{table-2-bubble-res-freqs} shows that the normal modes are quite close to the single bubble resonant frequency in this regime.

The final regime occurs when the separation distance becomes very large compared to  the radius of the bubbles. In this situation the sensitivity of the Hankel function in the layer potentials to negative imaginary numbers becomes apparent, leading to a much wider variation in the eigenvalues of $\mathcal{A}_2$. Similarly to when the bubbles are very close together, we observe significant resonant frequency shifts in this regime. When $d$ varies from $100$ to $1000$ we obtain the spectrum shown in Figure \ref{fig-bubble-regime-far}. Here we have a symmetric mode $\omega_s \approx 0.0013 - 0.00577i$ and an antisymmetric mode $\omega_a \approx 0.0308 - 0.00575i$.

\begin{figure}
  \centering
  \vertfig
    [The distance between the bubbles is varying from $0.1$ to $1$.]
    {\includegraphics[width=1.0\textwidth]{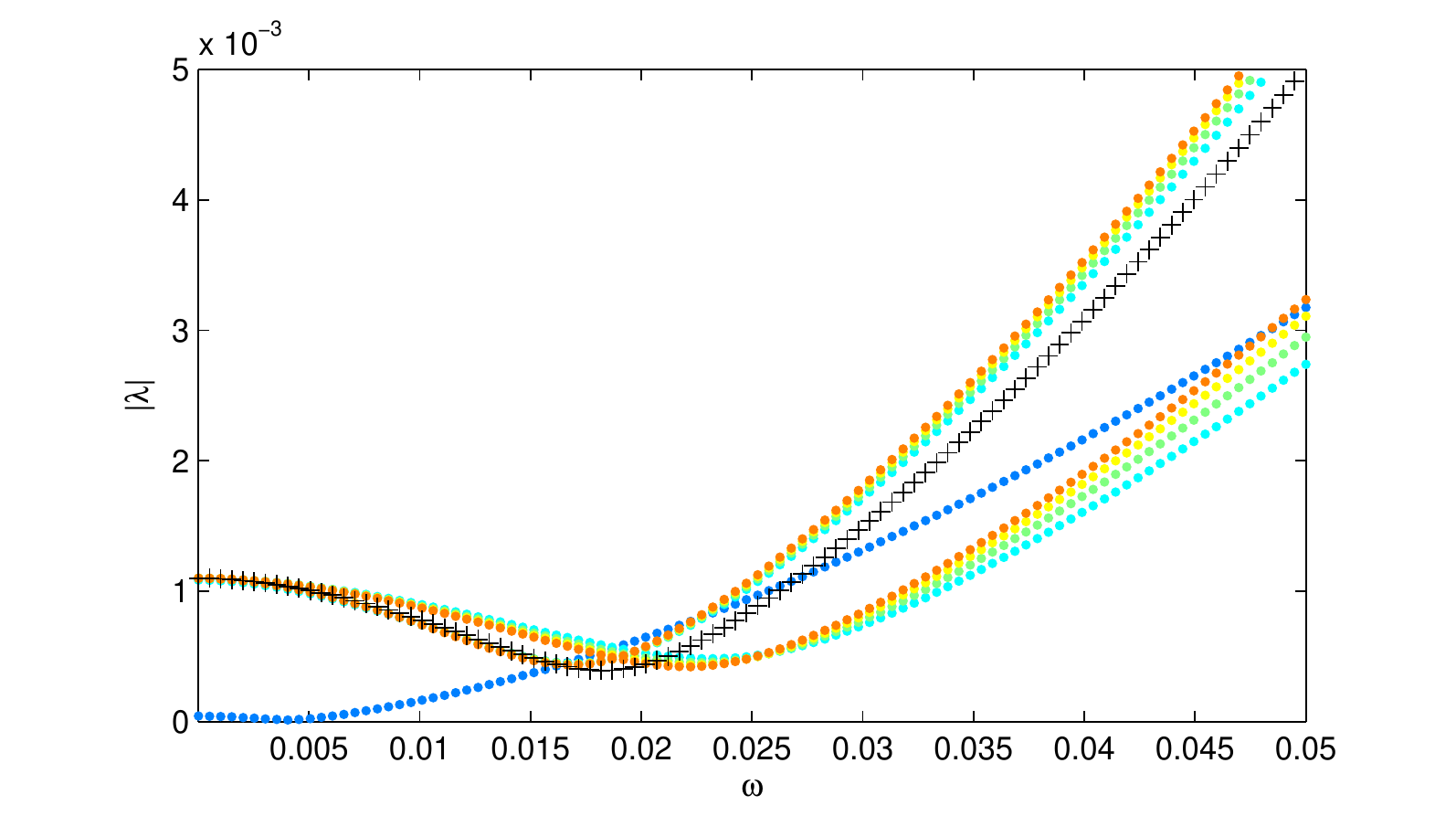}
    \label{fig-2-bubbles-dist-01-1}}
  \vertfig
    [The distance between the bubbles is varying from $10$ to $100$.]
    {\includegraphics[width=1.0\textwidth]{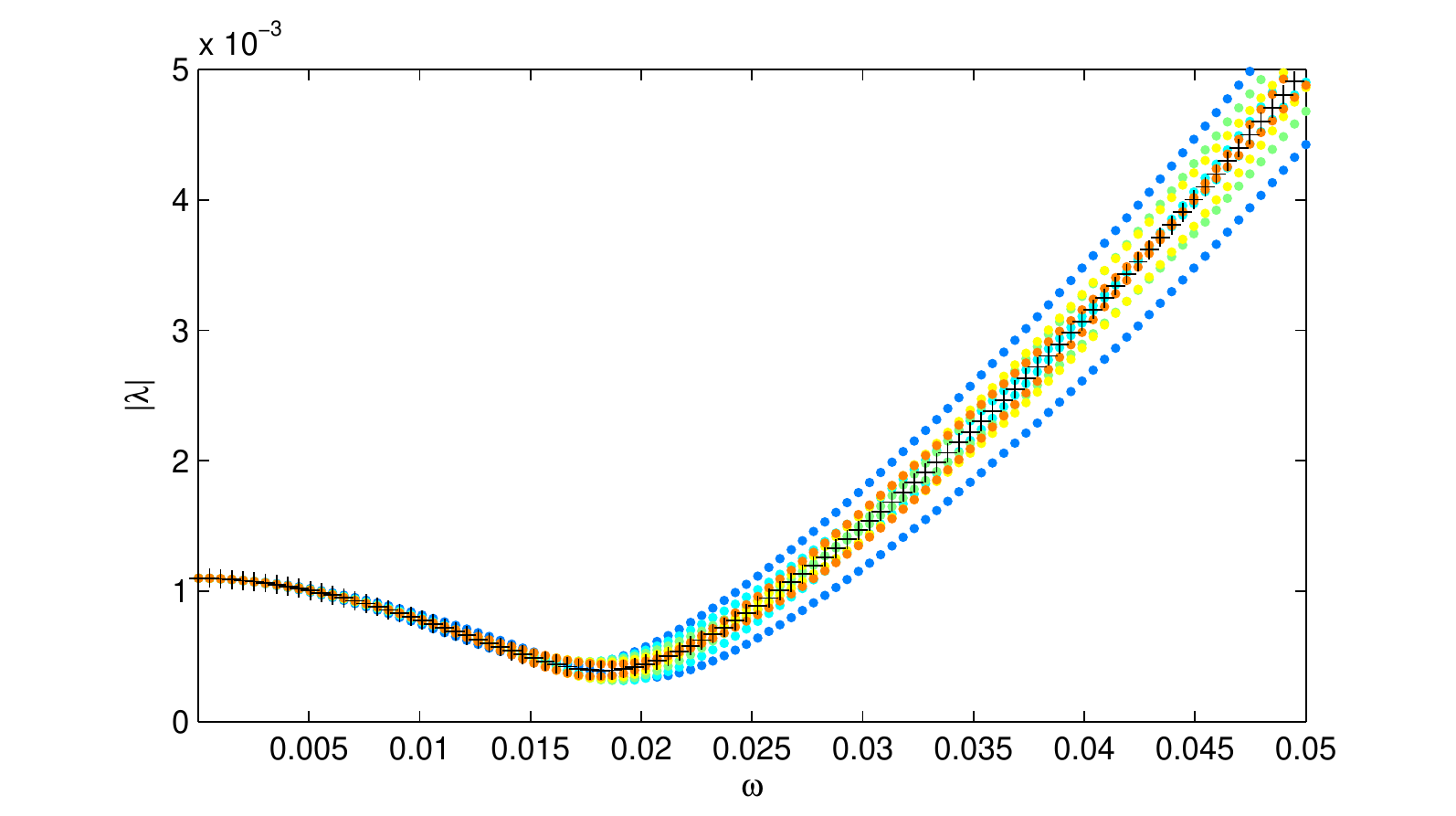}
    \label{fig-2-bubbles-dist-10-100}}
  \caption{$|\lambda|$ when $\omega \in \R$ for $\lambda \in \sigma(\mathcal{A})$ (black crosses) and $\lambda \in \sigma(\mathcal{A}_2)$ (colored dots) . The distance increases as the dots change from blue to orange. Although the eigenvalues of $\mathcal{A}_2$ approach those of $\mathcal{A}$ as the distance increases, they don't go to zero when $\omega$ is real. Here, $\sigma(\mathcal{A})$ and $\mathcal{A}_2$  are the spectra of $\mathcal{A}$ and $\sigma(\mathcal{A}_2)$, respectively. } \label{fig-bubble-regime-mid-real}
\end{figure}

\begin{figure}
  \centering
  \vertfig
    [The distance between the bubbles is 10.]
    {\includegraphics[width=1.0\textwidth]{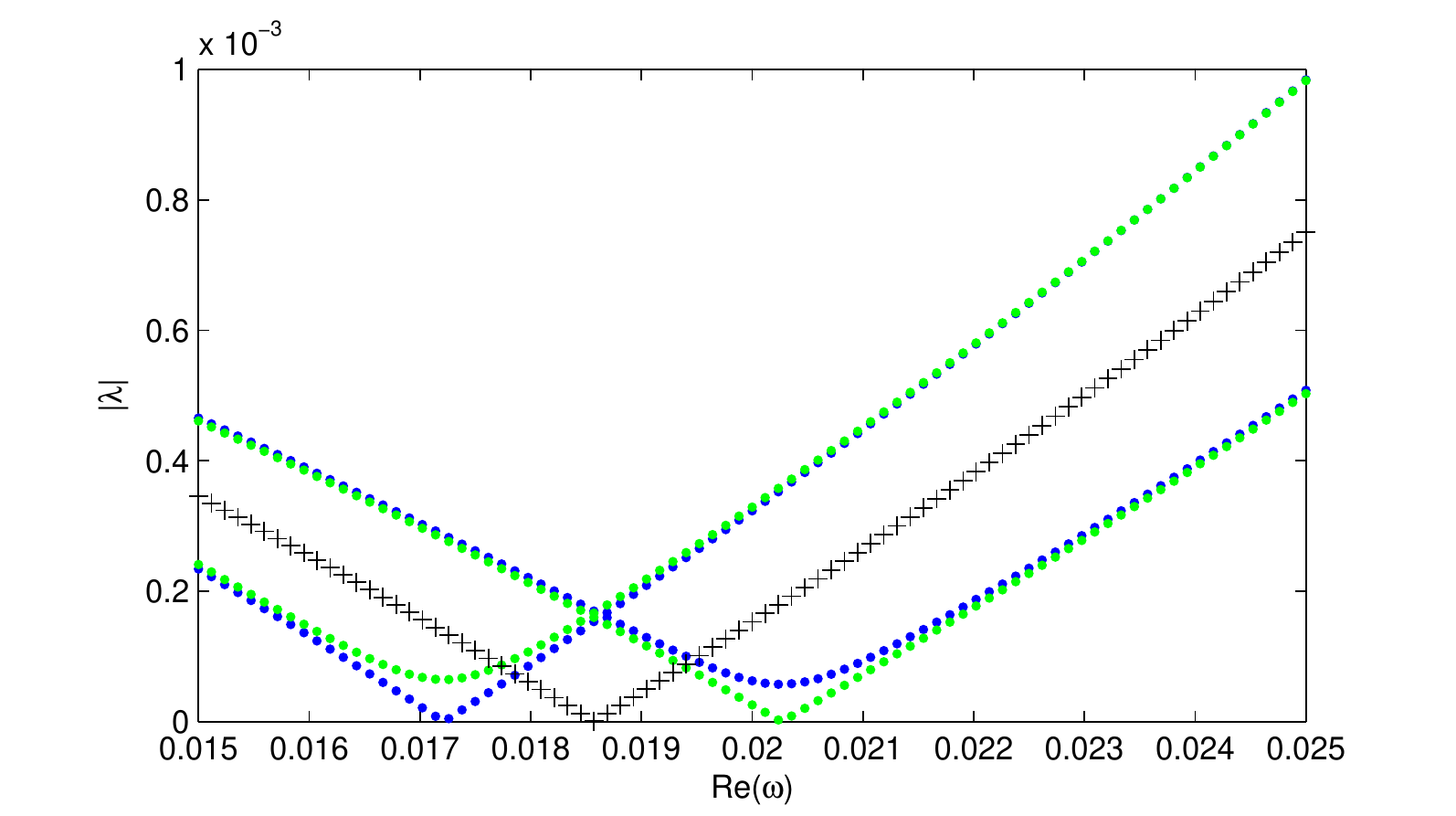}\label{fig-bubble-regime-mid-complex-10}}
  \vertfig
    [The distance between the bubbles is 100.]
    {\includegraphics[width=1.0\textwidth]{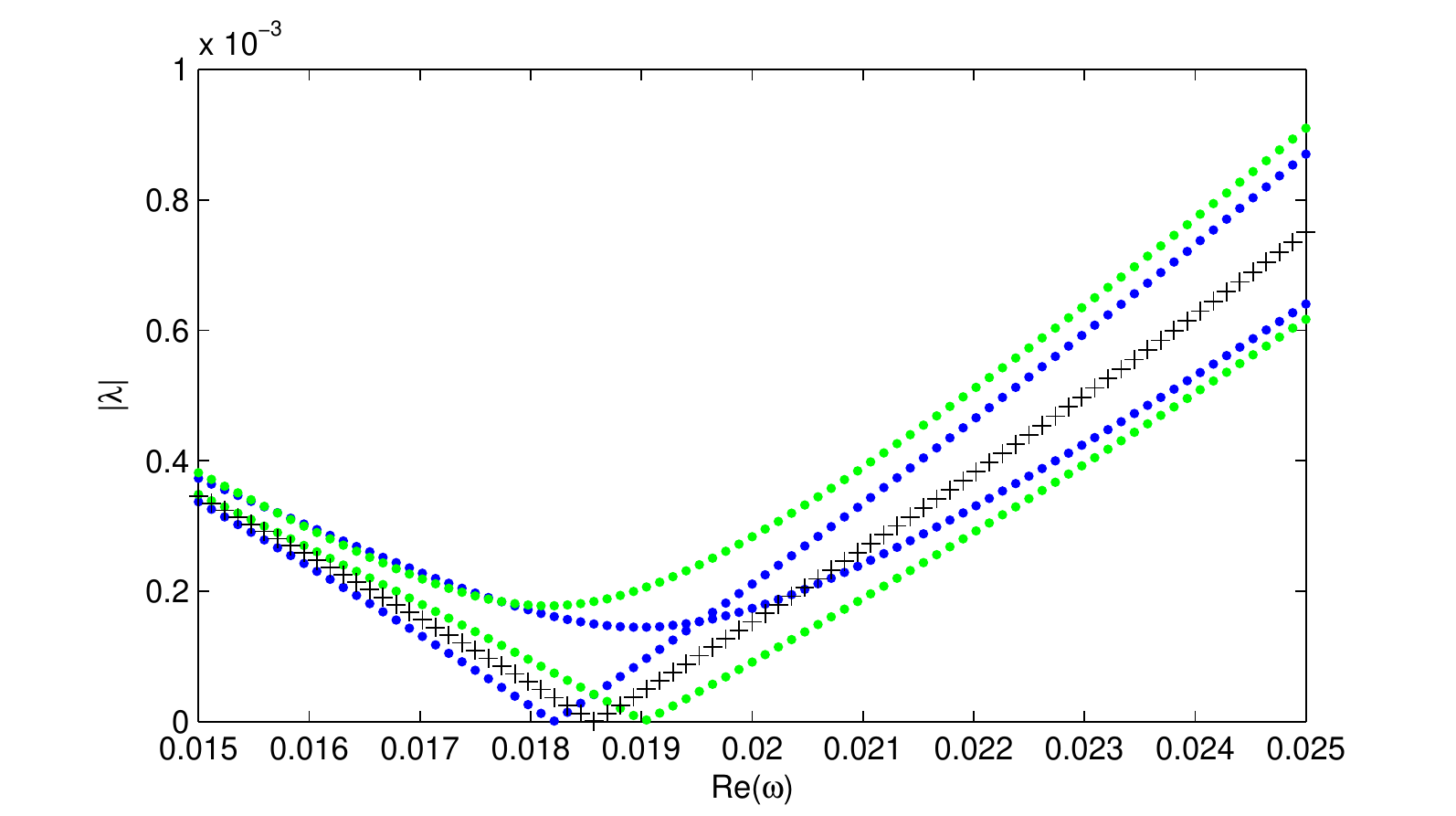}\label{fig-bubble-regime-mid-complex-100}}
  \caption{The eigenvalues of $\mathcal{A}$ (black crosses) and $\mathcal{A}_2$ (blue and green dots) may go to zero in the regime where the bubbles are a moderate distance apart, provided $\omega$ has some negative imaginary part. The frequency shift is less pronounced when $d = 100$ as opposed to $d = 10$ due to the decrease in the interaction of the bubbles with each other.}
  \label{fig-bubble-regime-mid-complex}
\end{figure}

\setlength{\tabcolsep}{12pt}
\begin{table}
    \centering
    \begin{tabular}{CCCC}
        \toprule
        & d = 10 & d = 100  \\        
        \midrule
        \omega_s & 0.01722793 -0.00407516i & 0.01819212 -0.00316674i  \\
        \omega_a & 0.02025476 -0.00349214i & 0.01905723 -0.00470526i  \\
        \bottomrule
    \end{tabular}
    \caption{The normal modes of the two bubble system shown in Figure \ref{fig-bubble-regime-mid-complex}. They are quite close to the resonant frequency of a single bubble in this regime, in contrast to the strong frequency shifts observed when $d \ll a$ and $d \gg a$.}
    \label{table-2-bubble-res-freqs}
\end{table}

 \begin{figure}
  \centering
\includegraphics[width=1.0\textwidth]{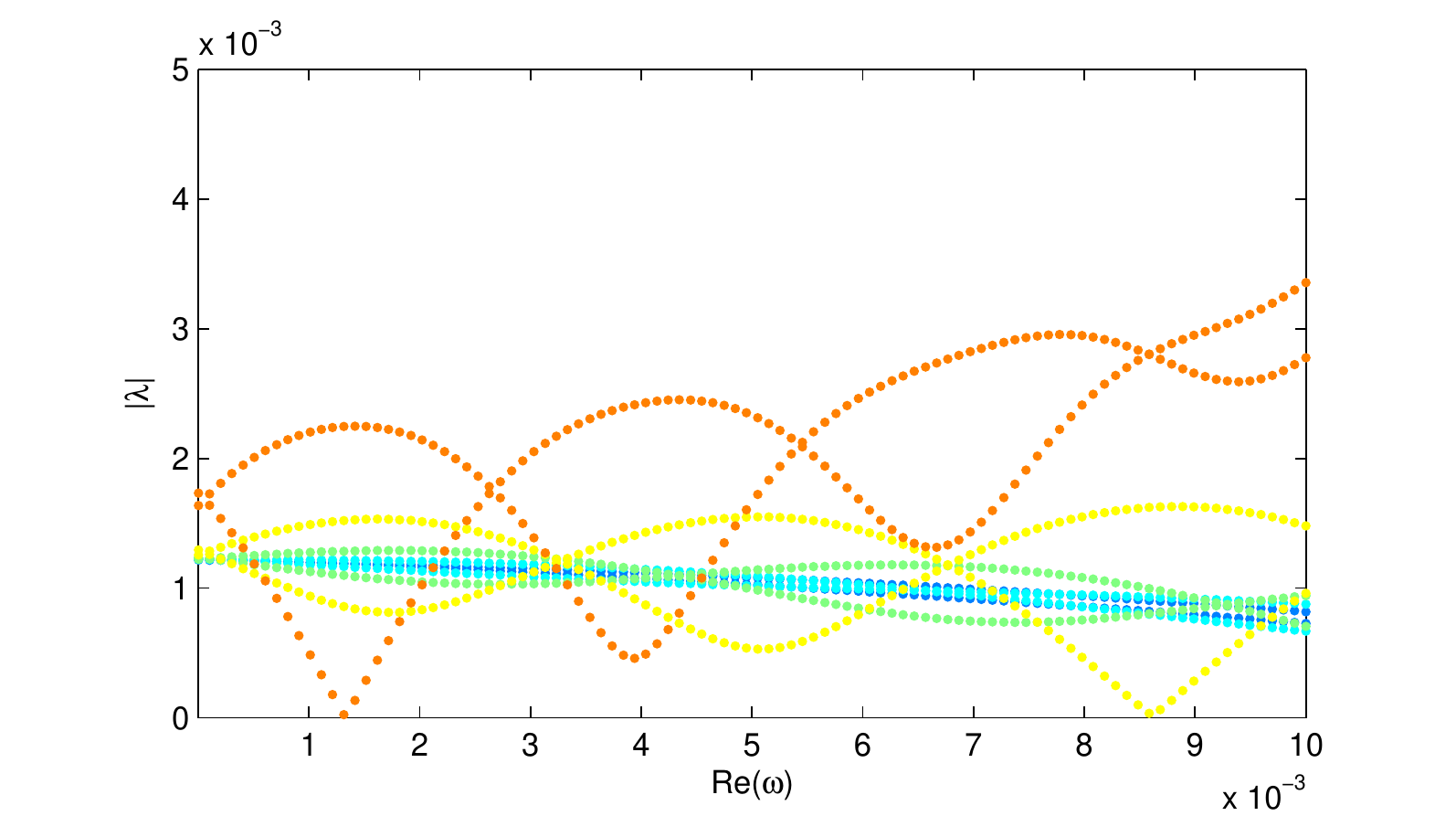}
  \caption{The sensitivity of the Hankel function in the layer potentials to negative imaginary numbers is apparent when the distance between the bubbles is very large. This leads to a signification reduction in the real part of the resonant frequencies. Here $d$ varies from $100$ to $1000$ and $\Im(\omega) = -0.00577i$. We have a symmetric mode at $\omega_s \approx 0.0013 - 0.00577i$.}
  \label{fig-bubble-regime-far}
\end{figure}

\section{Concluding remarks}
In this paper we have investigated the acoustic wave propagation problem in bubbly media and for the first time rigorously derived the low frequency resonances. Furthermore, we have provided a justification for the monopole approximation. The techniques developed in this paper open a door for a mathematical and numerical framework for investigating acoustic wave propagation in bubbly media. In forthcoming papers we will investigate the superabsorption effect that can be achieved using bubble metascreens \cite{leroy2, leroy3}. We will also mathematically justify 
Foldy's approximation and quantify time-reversal and the superfocusing effect  in bubbly media probed at their Minnaert resonant frequency \cite{lanoy}. Finally, we will develop accurate and fast numerical schemes for solving acoustic wave propagation problems in the presence of   closely spaced bubbles. 

\appendix
\section{Some asymptotic expansions} \label{sec-appendix-3d}

We recall some basic asymptotic expansion for the layer potentials in three and two dimensions from \cite{akl}; see also the appendix in \cite{matias}. 

\subsection{Some asymptotic expansions in three dimensions}
We first consider the single layer potential:
$$
\mathcal{S}_{D}^{k} [\psi](x) =  \int_{\p D} G(x, y, k) \psi(y) d\sigma(y),  \quad x \in  \p {D},
$$
where $$G(x, y, k)= - \f{e^{ik|x-y|}}{4 \pi|x-y|}.$$
We have the following asymptotic expansion:
\be \label{series-s}
\mathcal{S}_{D}^{k}=  \mathcal{S}_{D} + \sum_{j=1}^{\infty} k^j \mathcal{S}_{D, j},
\ee
where
$$
\mathcal{S}_{D, j} [\psi](x) = - \f{i}{4 \pi} \int_{\p D} \f{ (i|x-y|)^{j-1}}{j! } \psi(y)d\sigma(y).
$$
In particular, we have
\bea
\mathcal{S}_{D} [\psi](x) &=& -   \int_{\p D} \f{1}{4 \pi|x-y|} \psi(y)d\sigma(y), \\
\mathcal{S}_{D, 1} [\psi](x) &=& - \f{i}{4 \pi}  \int_{\p D}  \psi(y)d\sigma(y), \\
\mathcal{S}_{D, 2} [\psi](x) &=& - \f{1}{8 \pi}  \int_{\p D}  |x-y| \psi(y)d\sigma(y).
\eea

\begin{lem} \label{lem-appendix11} The norm
$\| \mathcal{S}_{D, j} \|_{
\mathcal{L}(L^2(\p D), H^1(\p D))}$ is uniformly bounded with respect to $j$. Moreover, the series in (\ref{series-s}) is convergent in 
$\mathcal{L}(L^2(\p D), H^1(\p D))$.
\end{lem}

We now consider the boundary integral operator $\mathcal{K}_{D}^{k, *}$ defined by
$$
\mathcal{K}_{D}^{k, *} [\psi](x)  = \int_{\p D } \f{\p G(x, y, k)}{ \p \nu(x)} \psi(y) d\sigma(y) ,  \quad x \in \p D.
$$
 We have
\be \label{series-k}
\mathcal{K}_{D}^{k,*}  = \mathcal{K}_D^* + k \mathcal{K}_{D, 1} + k^2 \mathcal{K}_{D, 2} + \ldots,
\ee
where
$$
\mathcal{K}_{D, j}[\psi](x) = - \f{i}{4 \pi} \int_{\p D} \f{ \p (i|x-y|)^{j-1}}{j! \p \nu(x)} \psi(y) d\sigma(y)=
- \f{i^j (j-1)}{4 \pi j!} \int_{\p D} |x-y|^{j-3} (x-y)\cdot\nu(x) \psi(y) d\sigma(y).
$$
In particular, we have
\beas
\mathcal{K}_{D, 1} &=&0,\\
\mathcal{K}_{D, 2}[\psi](x) &=& \f{1}{8\pi} \int_{\p D} \f{(x-y)\cdot \nu(x)}{|x-y|} \psi(y)d\sigma(y),\\
\mathcal{K}_{D, 3}[\psi](x) &=& \f{i}{12\pi} \int_{\p D} (x-y)\cdot \nu(x)\psi(y)d\sigma(y).
\eeas

\begin{lem} \label{lem-appendix13} The norm
$\| \mathcal{K}_{D, j} \|_{ \mathcal{L}(L^2(\p D))}$ is uniformly bounded for $j \geq 1$.
Moreover, the series in (\ref{series-k}) is convergent in $\mathcal{L}(L^2(\p D))$.
\end{lem}

\begin{lem}
The following identities hold: 
\begin{enumerate}
\item[(i)]
$$
\mathcal{K}_{D, 2}^*[\chi_{\p D}] (x)= \f{1}{8\pi}\int_{\p D} \f{(y-x)\cdot \nu(y)}{|y-x|}d\sigma(y)
= \f{1}{8\pi}\int_{D} \nabla \cdot \f{y-x}{|y-x|}dy =  \f{1}{4\pi}\int_{D} \f{1}{|y-x|}dy. 
$$
\item[(ii)]
$$
\mathcal{K}_{D, 3}^*[\chi_{\p D}](x) = \f{-i}{12 \pi}\int_{\p D} (y-x)\cdot{\nu(y)}d\sigma(y)
= \f{-i}{12 \pi}\int_{D} \nabla \cdot (y-x) dy = \f{-i}{12 \pi} 3 Vol (D) = \f{-i}{4\pi} Vol (D).
$$
\end{enumerate}
\end{lem}

\subsection{Some asymptotic expansions in two dimensions} \label{appendix-2d-2}

In two dimensions, the single-layer potential for the Helmholtz equation is defined by
$$
\mathcal{S}_{D}^{k} [\psi](x) =  \int_{\p D} G(x, y, k) \psi(y) d\sigma(y),  \quad x \in  \p {D},
$$
where $G(x, y, k)= -\dfrac{i}{4}H_0^{(1)}(k|x-y|)$ and $H_0^{(1)}$ is the Hankel function of first kind and order $0$. We have
$$
-\dfrac{i}{4}H_0^{(1)}(k|x-y|) = \dfrac{1}{2\pi}\ln |x-y|+\eta_k+\sum_{j=1}^{\infty}(b_j\ln k|x-y|+c_j)(k|x-y|)^{2j},
$$
where
$$
\eta_k = \dfrac{1}{2\pi}(\ln k+\gamma-\ln 2)-\dfrac{i}{4}, \quad b_j = \dfrac{(-1)^j}{2\pi}\dfrac{1}{2^{2j}(j!)^2}, \quad c_j = b_j\left(\gamma -\ln 2-\dfrac{i\pi}{2}-\sum_{n=1}^j\dfrac{1}{n}\right),
$$
and
$\gamma$ is the Euler constant. Especially, 
\[
b_1= -\f{1}{8 \pi}, \,\, c_1 = -\f{1}{8 \pi} (\gamma- \ln 2-1 -\f{i \pi}{2}).
\]
Thus, 
\be \label{series-s2d}
\mathcal{S}_{D}^{k}=  \hat{\mathcal{S}}_{D}^k +\sum_{j=1}^{\infty} \left(k^{2j}\ln k\right) \mathcal{S}_{D, j}^{(1)}+\sum_{j=1}^{\infty} k^{2j} \mathcal{S}_{D, j}^{(2)},
\ee
where
\be \label{hatcals}
\hat{\mathcal{S}}_{D}^k[\psi](x) = \mathcal{S}_{D}[\psi](x) + \eta_k\int_{\partial D} \psi \, d\sigma,\ee
and 
\beas
\mathcal{S}_{D, j}^{(1)} [\psi](x) &=& \int_{\p D} b_j|x-y|^{2j} \psi(y)d\sigma(y),\\
\mathcal{S}_{D, j}^{(2)} [\psi](x) &=& \int_{\p D} |x-y|^{2j}(b_j\ln|x-y|+c_j)\psi(y)d\sigma(y).
\eeas

We next consider the boundary integral operator 
$\mathcal{K}_{D}^{k, *}$ defined by
\[
\mathcal{K}_{D}^{k, *} [\psi](x) =\int_{\p D } \f{\p G(x, y, k)}{ \p \nu(x)} \psi(y) d\sigma(y) ,  \quad x \in \p D.
 \] 
 We have
\be \label{series-k2d}
\mathcal{K}_{D}^{k,*} = \mathcal{K}_D^* +\sum_{j=1}^{\infty} \left(k^{2j}\ln k\right) \mathcal{K}_{D, j}^{(1)}+\sum_{j=1}^{\infty} k^{2j} \mathcal{K}_{D, j}^{(2)},
\ee
where
\beas
\mathcal{K}_{D, j}^{(1)} [\psi](x) &=& \int_{\p D} b_j\dfrac{\partial |x-y|^{2j}}{\partial \nu(x)}\psi(y)d\sigma(y),\\
\mathcal{K}_{D, j}^{(2)} [\psi](x) &=& \int_{\p D} \dfrac{\partial \left( |x-y|^{2j}(b_j\ln|x-y|+c_j)\right)}{\nu(x)}\psi(y)d\sigma(y).
\eeas

\begin{lem}
The following estimates hold in $\mathcal{L}(L^2(\p D), H^1(\p D))$ and $\mathcal{L}(L^2(\p D), L^2(\p D))$, respectively: 
\beas
\mathcal{S}_{D}^{k}&=&  \hat{\mathcal{S}}_{D}^k +k^{2}\ln k \mathcal{S}_{D, 1}^{(1)}+k^{2} \mathcal{S}_{D, 1}^{(2)} + O(k^4 \ln k);\\
\mathcal{K}_{D}^{k,*}&=& \mathcal{K}_{D} +k^{2}\ln k \mathcal{K}_{D, 1}^{(1)}+k^{2} \mathcal{K}_{D, 1}^{(2)} + O(k^4 \ln k).
\eeas
\end{lem}

\begin{lem}
The following identities hold: 
\begin{enumerate}
\item[(i)]
$$
(\mathcal{K}_{D, 1}^{(1)})^*[\chi_{\p D}] (x)= 
4 \bar{b}_{1} Vol(D)\chi_{\p D}(x);
$$
\item[(ii)]
$$
(\mathcal{K}_{D, 1}^{(2)})^*[\chi_{\p D}] (x)= 
(2 \bar{b}_{1} + 4 \bar{c}_{1}) Vol(D) \chi_{\p D}(x) + 
4 \bar{b}_{1} \int_{D} \ln |x-y| d y,
$$
where $\bar{b}_{1}$ and $\bar{c}_{1}$ are the complex conjugates of ${b}_{1}$ and ${c}_{1}$.
\end{enumerate}
\end{lem}
\begin{proof}
First, we have
\beas
(\mathcal{K}_{D, 1}^{(1)})^*[\chi_{\p D}] (x)&=& 
\bar{b}_{1} \int_{\p D} 2(y-x, \nu(y)) d \sigma(y) \\
&=& 
\bar{b}_{1} \int_{\p D} \f{ \p |y-x|^2}{\p \nu(y)}d \sigma(y) \\
&=& 
\bar{b}_{1} \int_{D} \Delta_y |y-x|^2 dy \\
&=& 4 \bar{b}_{1} Vol(D)\chi_{\p D}(x). 
\eeas
We now prove the second identity. We have
\beas
(\mathcal{K}_{D, 1}^{(2)})^*[\chi_{\p D}] (x)&=&  
 \int_{\p D} \f{ \p \left[ |y-x|^2 ( \bar{b}_1 \ln |x-y| + \bar{c}_1)\right]}{\p \nu(y)}d \sigma(y) \\
&=&  \int_{D} \Delta_y [ |y-x|^2 ( \bar{b}_1 \ln |x-y| + \bar{c}_1)]dy
 \\
&=& 4 \bar{c}_{1} Vol(D)\chi_{\p D}(x) + \bar{b}_1  \int_{D} \Delta_y [ |y-x|^2  \ln |x-y|]dy \\
&=& 4 \bar{c}_{1} Vol(D)\chi_{\p D}(x) + \bar{b}_1  \int_{D}4 \ln |x-y|]dy  + \bar{b}_1\int_{D} 2 dy + \bar{b}_1 \int_{D} |y-x|^2 \Delta \ln |y-x| dy \\
&=& (2 \bar{b}_{1} + 4 \bar{c}_{1}) Vol(D) \chi_{\p D}(x) + 
4 \bar{b}_{1} \int_{D} \ln |x-y| d y,
\eeas
where we have used the fact that
\[
\int_{D} |y-x|^2 \Delta \ln |y-x| dy  =0, \,\, \mbox{for $x \in \p D$ }.  
\]
This completes the proof of the Lemma. 
\end{proof}

\section{The Minnaert resonance in two dimensions} \label{appendixB} 
In this section, we derive the Minnaert resonance for a single bubble in two dimensions using the same method we developed for the three-dimensional case. The main differences between the two-dimensional case and the three-dimensional case are as follows: (1) the single layer potential $\mathcal{S}_D$ may not be invertible from $L^2(\p D)$ to $H^1(\p D)$ in two dimensions, while this property always holds in three dimensions. We refer to \cite{akl, verchota} for more detail on this issue; (2) there is a logarithmic singularity in the asymptotic expansion of the single layer potential $\mathcal{S}_{D}^k$ for small $k$. These create some difficulties which we address here.  

Recall that
\[
\mathcal{A}(\omega, \delta) = 
 \begin{pmatrix}
  \mathcal{S}_D^{k_b} &  -\mathcal{S}_D^{k}  \\
  -\f{1}{2}Id+ \mathcal{K}_D^{k_b, *}& -\delta( \f{1}{2}Id+ \mathcal{K}_D^{k, *})
\end{pmatrix},
\]
where the boundary integral operators $\mathcal{S}_D^{k}$ and $\mathcal{K}_D^{k, *}$ are defined in Section \ref{appendix-2d-2} together with their asymptotic expansions. 

We denote by
\be  \label{eq-A_0-2d}
\mathcal{A}_0:=  
 \begin{pmatrix}
  \hat{\mathcal{S}}_D^{k_b} &  -\hat{\mathcal{S}}_D^k \\
  -\f{1}{2}Id+ \mathcal{K}_D^{*}& 0
\end{pmatrix},
\ee
where $\hat{\mathcal{S}}_D^k$  (resp. $ \hat{\mathcal{S}}_D^{k_b}$) is defined by (\ref{hatcals}) (resp. with $k$ replaced by $k_b$). 

Note that the kernel space of the operator 
$-\f{1}{2}Id+ \mathcal{K}_D^{*}$ has dimension one. We chose $\psi_0$ to be the real-valued function in this kernel space which has unit norm in $L^2(\p D)$. We have $\mathcal{K}_D^{*}[\psi_0] = \f{1}{2} \psi_0$. One can show that
\be \label{gamma0}
\mathcal{S}_D[\psi_0] = \gamma_0 \chi_{\p D}
\ee
for some constant $\gamma_0$ (see \cite{akl, verchota}). Here and after, we also denote by $\phi_0= \chi_{\p D}$. There are two cases:
\begin{enumerate}
\item[(i)]
Case I: $\gamma_0 =0$. 
\item[(ii)]
Case II: $\gamma_0 \neq 0$. 
\end{enumerate}
In case I, it is clear that $\mathcal{S}_D$ is not invertible from $L^2(\p D)$ to $H^1(\p D)$. 
In case II, we can show that $\mathcal{S}_D$ is invertible from $L^2(\p D)$ to $H^1(\p D)$.

We remark that $ (\chi_{\p D}, \psi_0) \neq 0$. Indeed, assume on the contrary that $ (\chi_{\p D}, \psi_0) = 0$. Then 
$$
(\mathcal{S}_D[\psi_0], \psi_0) = \gamma_0(\chi_{\p D}, \psi_0) =0,
$$
which further implies that $\psi_0=0$.  This contradiction proves our assertion. 

\begin{lem} \label{lem-2d-1}
In both cases, the operator $\hat{\mathcal{S}}_{D}^k$ is invertible in $\mathcal{L}(L^2(\p D), H^1(\p D))$. 
\end{lem}

\begin{proof}
We first show that $\hat{\mathcal{S}}_{D}^k$ is injective. Assume 
that 
\[
\hat{\mathcal{S}}_{D}^k[y] = \mathcal{S}_{D}[y] + \eta_k (y, \chi_{\p D}) \chi_{\p D}=0 \quad \mbox{for some } y \in L^2(\partial D).
\]
In Case I, we have $\mathcal{S}_{D}[y] \perp \psi_0$ in $L^2(\p D)$, therefore,
$\eta_k (y, \chi_{\p D}) (\chi_{\p D}, \psi_0) =0$. Since $ (\chi_{\p D}, \psi_0) \neq 0$, we obtain $(y, \chi_{\p D})=0$. It follows that 
$\mathcal{S}_{D}[y]=0$. But this implies that $y=c\psi_0$ for some constant $c$. Using the condition $(y, \chi_{\p D})=0$ again, we derive $c=0$, which shows that $y=0$. 

In Case II, we have $\mathcal{S}_{D}[\psi_0] \neq 0$. Since 
$\mathcal{S}_{D}[y] = -\eta_k (y, \chi_{\p D}) \chi_{\p D}$, we see that
$y= c\psi_0$ for some constant $c$. Therefore, 
\[
\gamma_0 c+ \eta_k c (\psi_0, \chi_{\p D}) = c ( \gamma_0  +\eta_k (\psi_0, \chi_{\p D}) )=0. 
\]
Note that $ \gamma_0  +\eta_k (\psi_0, \chi_{\p D})  \neq 0$, which follows from the fact that both $\gamma_0$ and $(\psi_0, \chi_{\p D})$ are real numbers while $\eta_k$ is a complex number with nonzero imaginary part. Thus we have $c=0$, and $y=0$ follows immediately. 

The surjectivity of $\hat{\mathcal{S}}_{D}^k$ follows from the fact that $\hat{\mathcal{S}}_{D}^k$ is Fredholm with index zero. This completes the proof of the lemma.  
 
\end{proof}

We have the following properties for the operator $\mathcal{A}_0$.

\begin{lem} \label{lem-2d-2}
We have
\begin{enumerate}
\item[(i)]
$ Ker (\mathcal{A}_0) = span\, \{ \Psi_0 \}$ where 
\[
\Psi_0 = \alpha_0 \begin{pmatrix}
    \psi_0\\
  a\psi_0\end{pmatrix}
\]
with 
\[
a = \begin{cases} 
\ds \f{\eta_{k_b}}{\eta_k}, \,\, &\mbox{in Case I}, \\ 
\nm
\ds
\f{\gamma_0 + (\psi_0, \phi_0)\eta_{k_b}}{\gamma_0 + (\psi_0, \phi_0)\eta_{k}},\,\, &\mbox{in Case II},
\end{cases} 
\]
 and the constant $\alpha_0$ being chosen such that $\|\Psi_0\|=1$;

\item[(ii)]
$ Ker (\mathcal{A}_0^*) = span\, \{ \Phi_0 \}$ where 
\[
\Phi_0 = \beta_0 \begin{pmatrix}
    0\\
  \phi_0\end{pmatrix}
\]
with $\phi_0 = \chi_{\p D}$ and the constant $\beta_0$ being chosen such that $\|\Phi_0\|=1$.
\end{enumerate}
\end{lem}

\begin{proof}
We first find the kernel space of $\mathcal{A}_0$. Assume that
\[
 \mathcal{A}_0  
\begin{pmatrix}
  y_b\\
  y
\end{pmatrix}
= \begin{pmatrix}
  \hat{\mathcal{S}}_D^{k_b}[y_b] - \hat{\mathcal{S}}_D^{k} [y] \\
  (-\f{1}{2}Id + \mathcal{K}_D^*)[y_b] 
  \end{pmatrix}
  =0 \quad \mbox{for some } y, y_b \in L^2(\partial D).
\]
We have
\bea
\mathcal{S}_D [y_b-y]  + \eta_{k_b} (y_b, \chi_{\p D}) \chi_{\p D}-
  \eta_{k} (y, \chi_{\p D}) \chi_{\p D} &=&0, \label{eq-11}\\
 (-\f{1}{2}Id + \mathcal{K}_D^*)[y_b]  &=&0. \label{eq-12}
\eea  
From (\ref{eq-12}), we see that $y_b$ is a multiple of $\psi_0$. We let
$y_b= \psi_0$. We now find the function $y$.
 
In Case I, we have $\mathcal{S}_D [y_b-y] \perp \psi_0$. Similarly to the proof in Lemma \ref{lem-2d-1}, we can derive that
$ y=c\psi_0$ for some constant $c$ which satisfies 
$$
\eta_{k_b} (\psi_0, \chi_{\p D}) -
  \eta_{k} c(\psi_0, \chi_{\p D})=0.
$$  
Thus, it follows that $c= {\eta_{k_b}}/{\eta_k}$.

In Case II, $\mathcal{S}_D$ is invertible. From (\ref{eq-11}), we can derive that $\psi_0- y$ is a multiple of $\psi_0$, which further implies that $y= c \psi_0$ for some constant $c$. Plugging this back to (\ref{eq-11}), we obtain
$$
(1-c)\gamma_0 + \eta_{k_b} (\psi_0, \chi_{\p D}) - \eta_{k} c(\psi_0, \chi_{\p D}) =0.
$$
Therefore, 
$$
c=\f{\gamma_0 + (\psi_0, \phi_0)\eta_{k_b}}{\gamma_0 + (\psi_0, \phi_0)\eta_{k}}.
$$
Note that $\gamma_0 + (\psi_0, \phi_0)\eta_{k} \neq 0$ because the $\eta_k$ has nonzero imaginary part. This completes the proof of the first part of the Lemma. 

The second part of the Lemma follows easily from the fact that the operator 
$\hat{\mathcal{S}}_D^{k}$ is injective. This complete the proof of the Lemma. 
\end{proof}

We next perform an asymptotic analysis in terms of $\delta$ and $\omega$ of the operator $\mathcal{A}(\omega, \delta)$. 

\begin{lem} In the space
$\mathcal{L}(\mathcal{H}, \mathcal{H}_1)$, we have
\[
\mathcal{A}(\omega, \delta):=\mathcal{A}_0 + \mathcal{B}(\omega, \delta)
= \mathcal{A}_0 + \omega^2 \ln \omega \mathcal{A}_{1,1, 0}+ \omega^2 \mathcal{A}_{1, 2, 0}
+ \delta \mathcal{A}_{0, 1}+ O(\delta \omega^2 \ln \omega)  + O(\omega^4 \ln \omega),
\]
where 
\[
\mathcal{A}_{1, 1,0}= 
\begin{pmatrix}
  v_b^2\mathcal{S}_{D,1}^{(1)} &  -v^2\mathcal{S}_{D,1}^{(1)}  \\
  v_b^2\mathcal{K}_{D,1}^{(1)}& 0
\end{pmatrix},
\,\, 
\mathcal{A}_{1, 2,0}= 
\begin{pmatrix}
  v_b^2\left( \ln v_b \mathcal{S}_{D,1}^{(1)} + \mathcal{S}_{D,1}^{(2)}\right) &  -  v^2\left( \ln v \mathcal{S}_{D,1}^{(1)} + \mathcal{S}_{D,1}^{(2)}\right)  \\
   v_b^2\left( \ln v_b \mathcal{K}_{D,1}^{(1)} + \mathcal{K}_{D,1}^{(2)}\right)& 0
\end{pmatrix},
\]
and 
\[
\mathcal{A}_{0, 1}=
\begin{pmatrix}
0& 0\\
0 &  -(\f{1}{2}Id+ \mathcal{K}_{D}^*)
\end{pmatrix}.
\]
\end{lem}

We define a projection $\mathcal{P}_0$ by 
$$
\mathcal{P}_0[\Psi]:= (\Psi, \Psi_0)\Phi_0,
$$
and denote by
$$
\tilde{\mathcal{A}_0}= \mathcal{A}_0 + \mathcal{P}_0.
$$

With the help of Lemma \ref{lem-2d-1}, we can establish the following results.
\begin{lem} We have
\begin{enumerate}
\item[(i)]
The operator $\tilde{\mathcal{A}_0}$ is a bijective operator in
$\mathcal{L}(\mathcal{H}, \mathcal{H}_1)$. Moreover, 
$\tilde{\mathcal{A}_0}[\Psi_0]= \Phi_0 $;
\item[(ii)]
$\tilde{\mathcal{A}_0}^*$ is a bijective operator in
$\mathcal{L}(\mathcal{H}, \mathcal{H}_1)$. Moreover, 
$\tilde{\mathcal{A}_0}^*[\Phi_0] = \Psi_0$. 
\end{enumerate}
\end{lem}

Our main results in two dimensions are summarized in the following theorem. 
\begin{thm} \label{thm-resonance-2d}
In the  quasi-static regime, there exist resonances (or the Minnaert resonance) for a single bubble. Their leading order terms are given by the roots of the following equation:

\bea \label{eq-resonance-2d}
\omega^2 \ln \omega +
\left[(\ln v_b + 1 + \f{c_1}{b_1}) - \f{\gamma_0 }{(\psi_0, \chi_{\p D})} \right] \omega^2 
- \f{1}{4Vol(D)} \f{a \delta}{b_1} =0, 
\eea
where the constants $b_1, c_1$ are defined in Section \ref{appendix-2d-2}, $\gamma_0$ in (\ref{gamma0}) and $a$ in Lemma \ref{lem-2d-2}.
\end{thm}

\begin{proof}
As in Theorem \ref{thm-resoance}, 
we can show that the resonances are the roots of the following equations 
\[
A(\omega, \delta):= \left((\tilde{\mathcal{A}_0} + \mathcal{B})^{-1} [\Phi_0], \Psi_0\right) - 1=0.
\]
By a direct calculation, we further have
\beas
A(\omega, \delta)
&=& -\omega^2 \ln \omega \left( \mathcal{A}_{1,1, 0}[\Psi_0], \Phi_0\right)
-\omega^2 \left( \mathcal{A}_{1, 2,0}[\Psi_0], \Phi_0\right)
 \\
&&  -\delta \left( \mathcal{A}_{0,1}[\Psi_0], \Phi_0\right)
+ O(\omega^4 \ln \omega)+ O(\delta \omega^2 \ln \omega). 
\eeas

It is clear that
\beas
(\mathcal{A}_{1,1, 0})^*[\Phi_0] &=&
\begin{pmatrix}
   \beta_0 v_b^2 (\mathcal{K}_{D, 1}^{(1)})^* [\chi_{\p D}]\\
    0 
\end{pmatrix}, \\ 
(\mathcal{A}_{1,2, 0})^*[\Phi_0] &=&
\begin{pmatrix}
   \beta_0 v_b^2 [\ln v_b (\mathcal{K}_{D, 1}^{(1)})^* [\chi_{\p D}]
   +\mathcal{K}_{D, 1}^{(2)})^*[\chi_{\p D}] \\
    0 
\end{pmatrix}, \\ 
\mathcal{A}_{0,1}[\Psi_0] &=&
\begin{pmatrix}
    0 \\
   -\alpha_0 v_b^2 (\f{1}{2}Id + \mathcal{K}_{D}^*) [a \psi_0]
\end{pmatrix} =
\begin{pmatrix}
    0 \\
   -\alpha_0 a v_b^2 \psi_0
\end{pmatrix}.
\eeas
It follows that
\beas
\left( \mathcal{A}_{1,1, 0}[\Psi_0], \Phi_0\right)&=&
\alpha_0 \beta_0( \psi_0,  v_b^2 (\mathcal{K}_{D, 1}^{(1)})^* \chi_{\p D}) = \alpha_0 \beta_0( \psi_0,  v_b^2 4 \bar{b}_{1} Vol(D)\chi_{\p D}) 
\\
&=&
4 \alpha_0 \beta_0 v_b^2 b_1 Vol(D) (\psi_0, \chi_{\p D});\\ 
\left( \mathcal{A}_{1,2, 0}[\Psi_0], \Phi_0\right)&=&
\alpha_0 \beta_0 \left( \psi_0,  v_b^2 [\ln v_b (\mathcal{K}_{D, 1}^{(1)})^* [\chi_{\p D}]
   +(\mathcal{K}_{D, 1}^{(2)})^*[\chi_{\p D}] \right) \\
&=&
4 \alpha_0 \beta_0 v_b^2 \ln v_b b_1 Vol(D) (\psi_0, \chi_{\p D})
+  \\
&&  \alpha_0 \beta_0 v_b^2 \left(\psi_0, (2 \bar{b}_{1} + 4 \bar{c}_{1}) Vol(D) \chi_{\p D}(x) + 
4 \bar{b}_{1} \int_{D} \ln |x-y| d y\right) \\
&=&
\alpha_0 \beta_0 v_b^2   Vol(D) (4 b_1 \ln v_b  + 4b_1 +4c_1) (\psi_0, \chi_{\p D}) + 4b_1\alpha_0 \beta_0 v_b^2 (\psi_0, \int_{D} \ln |x-y| d y) \\
&=&
4\alpha_0 \beta_0 v_b^2   Vol(D) (b_1 \ln v_b b_1 + b_1 + c_1) (\psi_0, \chi_{\p D}) - 4b_1\alpha_0 \beta_0 v_b^2 \gamma_0 Vol(D); \\
\left( \mathcal{A}_{0,1}[\Psi_0], \Phi_0\right)  &=&
- \alpha_0 \beta_0 a v_b^2 (\psi_0, \chi_0),
 \eeas
where we have used the fact
\beas
(\psi_0, \int_{D} \ln |x-y| d y) &=& \int_{\p D} \psi_0(x)d\sigma(x) \int_{\p D} \ln |x-y| d y
= \int_{D} dy \int_{\p D}  \ln |x-y| \psi_0(x) d\sigma(x) \\
&=& \int_{D} -\gamma_0 dy = -\gamma_0 Vol(D)  
\eeas
in the second equality above. 
Therefore, we derive that
\beas
& 4b_1 Vol(D) (\psi_0, \chi_{\p D}) \omega^2 \ln \omega +
4\bigg[Vol(D) (b_1\ln v_b + b_1 +c_1) (\psi_0, \chi_{\p D}) - b_1  \gamma_0 Vol(D) \bigg] \omega^2 \\
&
- a \delta(\psi_0, \chi_0) + O(\omega^4 \ln \omega)+ O(\delta \omega^2 \ln \omega) =0.
\eeas

This completes the proof of the lemma. 
\end{proof}

\begin{rmk}


In the special case when $D$ is the unit disk, we have Vol(D) = $\pi$ and $\gamma_0=0$. Therefore, the Minnaert resonance in two dimensions is given by the roots of the following equation:
\bea \label{eq-resonance-2d-unit-circle}
\omega^2 \ln \omega + (\ln v_b + 1 + \f{c_1}{b_1})\omega^2 
- \f{1}{4\pi}\f{a \delta}{b_1} = 0. 
\eea
\end{rmk}

\begin{rmk}
We can use the same method as in Section \ref{sec-3d-point-scatter} to derive the point scatterer approximation for the scattering by a single bubble in two dimensions. 
\end{rmk}


\begin{thebibliography}{99}

\bibitem{inv} {H. Ammari, G. Ciraolo, H. Kang, H. Lee, and G. Milton}, Spectral theory of a Neumann-Poincar\'e-type operator and analysis of
cloaking due to anomalous localized resonance, Arch. Rational Mech. Anal.,  208 (2012), 667--692.


\bibitem{book2} {H. Ammari and H. Kang}, \textsl{Polarization and Moment
Tensors with Applications to Inverse Problems and Effective Medium
Theory}, Applied Mathematical Sciences, Vol. 162, Springer-Verlag,
New York, 2007.

\bibitem{AK04} {H. Ammari and H. Kang}, Boundary layer techniques for solving the Helmholtz equation
in the presence of small inhomogeneities, J. Math. Anal. Appl., 296 (2004), 190--208.

\bibitem{akl} H. Ammari, H. Kang, and H. Lee, \textsl{Layer Potential Techniques in Spectral Analysis}, Mathematical surveys and monographs, Vol. 153, Amer. Math. Soc., Rhode Island, 2009. 


\bibitem{hai} H. Ammari and H. Zhang, A mathematical theory of super-resolution by using a system of sub-wavelength Helmholtz resonators, Comm. Math. Phys., 337 (2015),  379--428.

\bibitem{hai2} H. Ammari and H. Zhang, Super-resolution in high contrast media, Proc. Royal Soc. A, 2015 (471), 20140946. 

\bibitem{matias} H. Ammari, P. Millien, M. Ruiz, and H. Zhang, 
Mathematical analysis of plasmonic nanoparticles: the scalar  
case,  arXiv:1506.00866. 

\bibitem{caflish} R.E. Caflish, M.J. Miksis, G.C. Papanicolaou, and L Ting, Effective equations for wave propagation in bubbly liquids, J. Fluid Mech., 153 (1985), 259--273.

\bibitem{caflish2} R.E. Caflish, M.J. Miksis, G.C. Papanicolaou, and L Ting, Wave propagation in bubbly liquids at finite volume fraction, J. Fluid Mech., 160 (1985), 1--14.


\bibitem{calvo} D.C. Calvo, A.L. Thangawng, and C.N. Layman, 
Low-frequency resonance of an oblate spheroidal cavity in a soft elastic medium, J. Acoust. Soc. Am., 132 (2012), EL1--EL7. 

  
\bibitem{commander} K.W. Commander and A. Prosperetti, Linear pressure waves in bubbly liquids: Comparison between theory and experiments, J. Acoust. Soc. Am., 85 (1989), 732--746. 


\bibitem{leroy1} M. Devaud, T. Hocquet, J. C. Bacri, and V. Leroy, The Minnaert bubble: an acoustic approach, Europ. J. Phys., 29.6 (2008), 1263.


\bibitem{Foldy} L. L. Foldy, The multiple scattering of waves. I. General theory of isotropic scattering by randomly distributed scatterers, Phys. Rev., 67 (1945), 107.


\bibitem{breathing} V. Galstyan, O.S. Pak, and H.A. Stone, A note on the breathing mode of an elastic sphere in Newtonian and complex fluids, Phys. Fluids, 27 (2015), 032001.  

\bibitem{hwang} P.A. Hwang and W.J. Teague, Low-frequency resonant scattering of bubble clouds, J. Atmosphere Oceanic Tech., 17 (2000), 847--853. 

\bibitem{kargl} S.G. Kargl, Effective medium approach to linear acoustics in bubbly liquids, J. Acoust. Soc. Am., 111 (2002), 168--173.

\bibitem{damir} D.B. Khismatullin, Resonance frequency of microbubbles: Effect of viscosity, J.  Acoust. Soc. Am.,
 116 (2004), 1463--1473.

\bibitem{Lek} J. Lekner, Capacitance coefficients of two spheres, Jour. Electrostatics, 69 (2011), 11--14.


\bibitem{lanoy} M. Lanoy, R. Pierrat, F. Lemoult, M. Fink, V Leroy, A Tourin, Subwavelength focusing in bubbly media using broadband time reversal, Phys. Rev., B 91.22 (2015), 224202.


\bibitem{leroy2} V. Leroy, A. Bretagne, M. Fink, A. Tourin, H. Willaime and P. Tabeling, Design and characterization of bubble phononic crystals, Appl. Phys. Lett., 95.17 (2009), 171904.

\bibitem{unusual} V. Leroy, M. Devaud, and J.C. Bacri, The air bubble: experiments on an usual harmonic oscillator, Am. J. Phys., 10 (2002), 1012--1019.

\bibitem{leroy3} V. Leroy, A. Strybulevych, M. Lanoy, F. Lemoult, A. Tourin, and J. H. Page, Superabsorption of acoustic waves with bubble metascreens, Phys. Rev., B 91.2 (2015), 020301.

\bibitem{leroy4} V. Leroy, A. Strybulevych , M.G. Scanlon, and J.H. Page, Transmission of ultrasound through a single layer of bubbles, Europ. Phys. Jour., E 29.1 (2009), 123-130.

\bibitem{cheng} H. Cheng, W.Y. Crutchfield, M. Doery and L. Greengard, Fast, accurate integral equation methods for the analysis of photonic crystal fibers I: Theory, Optics Express 12.16 (2004), 3791-3805.

\bibitem{verchota} G. Verchota, Layer potentials and regularity for  the Dirichlet problem for Laplace's equation in Lipchitz domains, J. Funct. Anal., 59 (1984), 572-611.

\bibitem{feuillade} Feuillade, C, Scattering from collective modes of air bubbles in water and the physical mechanism of superresonances, J. Acoust. Soc. Amer., 98.2 (1995), 1178-1190.

\end{thebibliography}
\end{document}